\documentclass{amsart}
\usepackage{amsmath}
\usepackage{amssymb}
\usepackage{amsthm}
\usepackage{amsrefs}
\usepackage{amsxtra}
\usepackage{enumerate}
\usepackage[matrix,arrow,curve]{xy} 
\newdir^{ (}{{}*!/-.8ex/@^{(}}
\newdir_{ (}{{}*!/-.8ex/@_{(}}%

\ifdefined\acknowledgements
\else
\newenvironment{acknowledgements}{\par\textsc{Acknowledgements.}}{}
\fi

\let\tilde\widetilde

\newcommand{\Tan}{{\mathrm{T}}}
\newcommand{\loccit}{\emph{loc.\ cit.}}
\newcommand{\bX}{{X}}
\newcommand{\Z}{{\mathbb{Z}}}

\newcommand{\Q}{\mathbb{Q}}
\newcommand{\conn}{^\circ}
\newcommand{\sep}{^{\mathrm{sep}}}

\newcommand{\Hsc}{{H^{*\,\circ}}}
\newcommand{\Htsc}{\tilde{H}^{*\,\circ}}

\newcommand{\norm}[1][]{\NN_{#1}}
\newcommand{\dnorm}[1][]{{\widehat\NN_{#1}}}
\newcommand{\dnormst}[1][]{{\widehat\NN_{#1}^{\text{st}}}}
\newcommand{\normchar}[1][]{\NN^*_{#1}}
\newcommand{\dnormchar}[1][]{\widehat\NN^*_{#1}}
\newcommand{\normcochar}[1][]{\NN_{#1 *}}
\newcommand{\dnormcochar}[1][]{\widehat\NN_{#1 *}}

\DeclareMathOperator{\stab}{stab}
\DeclareMathOperator{\chr}{char}
\DeclareMathOperator{\diag}{diag}

\DeclareMathOperator{\Int}{Int}
\DeclareMathOperator{\Aut}{Aut}
\DeclareMathOperator{\Inn}{Inn}
\DeclareMathOperator{\GL}{GL}

\DeclareMathOperator{\Gal}{Gal}
\newcommand{\NN}{\mathcal{N}}

\newcommand{\TTst}{\mathcal{T}_{\mathrm{st}}}
\newcommand{\inv}{^{-1}}

\newcommand{\lsup}[1]{{}^{#1}}

\newcommand{\maaap}[4]{\ensuremath{{#2}\colon{#3}#1{#4}}}
\newcommand{\map}{\maaap\longrightarrow}
\newcommand{\mapto}{\maaap\longmapsto}   

\newcommand{\bimap}{\maaap{\stackrel{\sim}{\longrightarrow}}}
\newcommand{\abmap}[2]{\ensuremath{{#1}\longrightarrow{#2}}}

\newenvironment{condition}[1]{\vspace{.5\baselineskip}
\begin{center}
\begin{minipage}{.85\textwidth}
\par \textbf{#1.}}{\end{minipage}\end{center}\vspace{.5\baselineskip}}

\newtheorem{thm}{Theorem}[section] 
\newtheorem{lem}[thm]{Lemma}
\newtheorem{prop}[thm]{Proposition} 
\newtheorem{cor}[thm]{Corollary} 
\theoremstyle{definition} 
\newtheorem{defn}[thm]{Definition} 
\theoremstyle{remark} 
\newtheorem{example}[thm]{Example} 
\newtheorem{examples}[thm]{Examples} 
\newtheorem{rem}[thm]{Remark} 
\newtheorem{rems}[thm]{Remarks} 

\numberwithin{equation}{section}

\title[Lifting]%
{Lifting representations of finite reductive groups I: Semisimple conjugacy classes}

\date{\today}
\author{Jeffrey D.~Adler}
\email[Adler]{jadler@american.edu}
\author{Joshua M.~Lansky} 
\email[Lansky]{lansky@american.edu}
\address[Adler,Lansky]{Department of Mathematics and Statistics\\
American University\\
Washington, DC 20016-8050}
\thanks{%
Both authors were partially supported by 
the National Science Foundation (DMS-0854844),
the National Security Agency (H98230-07-1-002 and H98230-08-1-0068),
and 
Summer Faculty Research Awards
from the College of Arts and Sciences of American University.
}

\begin{document}

\begin{abstract}
Suppose that $\tilde{G}$ is a connected reductive group
defined over a field $k$,
and
$\Gamma$ is a finite group acting via $k$-automorphisms
of $\tilde{G}$ satisfying a certain quasi-semisimplicity condition.
Then the identity component of the group of $\Gamma$-fixed points in $\tilde{G}$
is reductive.
We axiomatize the main features of the relationship between this
fixed-point group and the pair $(\tilde{G},\Gamma)$,
and consider any group $G$ satisfying the axioms.
If both $\tilde{G}$ and $G$ are $k$-quasisplit, then we
can consider their duals $\tilde{G}^*$ and $G^*$.
We show the existence of and give an explicit formula for a natural map
from the set of semisimple stable conjugacy classes in $G^*(k)$
to the analogous set for $\tilde{G}^*(k)$.
If $k$ is finite, then our groups are automatically quasisplit,
and our result specializes to give a map
of semisimple conjugacy classes.
Since such classes parametrize packets of irreducible representations
of $G(k)$ and $\tilde{G}(k)$, one obtains a mapping of such packets.
\end{abstract}

\maketitle

\addtocounter{section}{-1}
\section{Introduction}
\subsection*{Motivation}
Suppose that $F$ is a $p$-adic field with residue field $k$;
$E/F$ is a finite, tamely ramified Galois extension;
$H$ is a connected, reductive $F$-group;
and $\tilde{H} = R_{E/F}H$
is formed from $H$ via restriction of scalars.
Then one expects to have
a \emph{base change lifting} that takes $L$-packets of smooth, irreducible representations
of $H(F)$ to $L$-packets for $H(E) = \tilde{H}(F)$.
We would like to gain an explicit understanding, in terms of compact-open data,
of base change for depth-zero representations,
and this problem requires us to construct a new lifting from (packets of) representations
of $G(k)$ to those of $\tilde{G}(k)$,
for various connected reductive $k$-groups $G$ and $\tilde{G}$ attached to parahoric subgroups
of $H(F)$ and $\tilde H(F)$, respectively.
Here $\Gamma=\Gal(E/F)$ acts on $\tilde{G}$, and the identity component of its group of fixed points is $G$.
(In most cases, this new lifting cannot itself be base change.
For more details, see~\cites{adler-lansky:bc-u3-ram,adler-lansky:group-actions}.)
Since representations of $G(k)$ can be parametrized by data associated
to the dual group $G^*$, it is enough to construct an appropriate lifting of such data
from $G^*$ to $\tilde{G}^*$.

\subsection*{This paper}
In the course of creating a candidate for this new lifting, we realized that we could work in greater
generality without much difficulty.
Namely, let $k$ denote an arbitrary field,
$\tilde{G}$ and $G$ connected reductive $k$-groups, and $\Gamma$ a finite group.
Instead of assuming that $G = (\tilde{G}^\Gamma)\conn$,
the identity component of the group of fixed points of $\Gamma$ in $\tilde{G}$,
we make the more general assumption that $G$ is a \emph{parascopic group} for $(\tilde{G},\Gamma)$
(see Definition \ref{defn:parascopic-group}).
There are several reasons to do this.
First, it clarifies our proofs.
Second, while our original motivation was to improve our explicit understanding
of base change for representations of $p$-adic groups,
we hope that our more general formulation will be applicable to the understanding
of a wider collection of correspondences of representations, including endoscopic transfer.
We will take this up elsewhere.


Under our hypotheses, we have the following results.
\begin{enumerate}[(A)]
\item
\label{stmt:geom-conj-lift}
Suppose $\tilde{G}$ and $G$ are $k$-quasisplit.
Then we obtain a natural map $\dnorm$
from the $k$-variety of semisimple geometric conjugacy classes of the dual $G^*$
to the analogous variety for $\tilde{G}^*$
(Proposition \ref{prop:lift-geometric-general}).
\item
\label{stmt:stable-lift}
By restricting and refining $\dnorm$, one obtains a map $\dnormst$
from the set of semisimple stable conjugacy classes
(in the sense of Kottwitz \cite{kottwitz:rational-conj})
in $G^*(k)$ to that in $\tilde{G}^*(k)$
(Theorem \ref{thm:main}).
\item
\label{stmt:rational-lift}
If $k$ is perfect of cohomological dimension $\leq 1$
(e.g., $k$ is finite),
then $\dnorm$ is a map from the set of semisimple conjugacy classes
in $G^*(k)$ to that in $\tilde{G}^*(k)$
(Corollary \ref{cor:k-finite}).
\end{enumerate}

Before proving the above results, we show that the situation that
motivated our notion of parascopy
really is a special case of it: that in which
$\Gamma$ acts on $\tilde{G}$ via $k$-automorphisms that all preserve
a common Borel subgroup of $\tilde{G}$ and a common maximal torus in the Borel subgroup,
and $G = (\tilde{G}^\Gamma)\conn$.
This ``Borel-torus pair'' need not be defined over $k$.
Under the above hypotheses, we prove a strong form of the following:
\begin{enumerate}[(A)]
\addtocounter{enumi}{3}
\item
\label{stmt:G-reductive}
$G$ is a reductive $k$-group
(Proposition \ref{prop:arb-gp-action-well-posed}).
\end{enumerate}

Although we assume that $\tilde{G}$ and $G$ are $k$-quasisplit
in Statements \eqref{stmt:geom-conj-lift} and \eqref{stmt:stable-lift}
(it is automatic in Statement \eqref{stmt:rational-lift}),
weaker hypotheses would suffice.
See Remark \ref{rem:nonquasisplit}.

\subsection*{Outline of this paper}
After establishing some basic notation (\S\ref{sec:setup}),
we consider in \S\ref{sec:modules}
how the action of a finite group on a torus $\tilde{T}$
gives rise to a norm map $\map\NN{\tilde{T}}{T}$
(where $T$ is the identity component of the group of fixed points of $\tilde{T}$),
and also corresponding maps on the modules
of characters and cocharacters of these tori.
In fact, we deal with the more general situation where
we may replace $T$ by any isogenous image of it
(see Condition P1  
of Definition \ref{defn:parascopy}).
We then prove a strong version of Statement \eqref{stmt:G-reductive} above (\S\ref{sec:group-actions}).
Suppose $G$ and $\tilde{G}$ are connected reductive $k$-groups, and $\Gamma$ is a finite group.
In \S\ref{sec:parascopy-defn},
we say what we mean by a \emph{parascopic datum} for the triple $(\tilde{G},\Gamma,G)$,
and say that $G$ is \emph{weakly parascopic}
for the pair $(\tilde{G},\Gamma)$
if such a datum exists.
Given such a datum, we have associated maximal tori $T\subset G$ and $\tilde{T}\subset\tilde{G}$,
an action of $\Gamma$ on the Weyl group $W(\tilde{G},\tilde{T})$,
and a canonical embedding $W(G,T) \longrightarrow W(\tilde{G},\tilde{T})^\Gamma$,
which we describe explicitly in \S\ref{sec:weyl-embedding}.
Using standard cohomological arguments, we classify in \S\ref{sec:stability}
the set of stable conjugacy classes of maximal $k$-tori in a reductive $k$-group.
In \S\ref{sec:parascopy-properties},
we call a weakly parascopic group $G$ \emph{parascopic}
if a compatibility condition between the maximal
$k$-tori of $G$ and $\tilde G$ is satisfied,
a condition that is automatic in many important cases (see
Examples~\ref{ex:parascopic-group}).
We then define and prove some basic results on \emph{equivalence} of parascopic data.
For example, if $(\phi,j_*)$ is a datum for $(\tilde{G},\Gamma, G)$
with respect to the maximal $k$-tori $\tilde{T}\subseteq \tilde{G}$ and $T\subseteq G$,
and $G$ is parascopic with respect to this datum,
then given any maximal $k$-torus $T'\subseteq G$,
there is an equivalent datum associated to $T'$.
This is crucial in showing that
all of our constructions are independent of the choice of a maximal $k$-torus in $G$.
If $G$ is quasisplit over $k$, then we can form its dual group $G^*$, and
we are then in a position to prove strong duality results (\S\ref{sec:duality})
between maximal $k$-tori in $G$ and in $G^*$.
In particular,
up to stable conjugacy, we have a canonical one-to-one correspondence of tori, 
and this correspondence preserves Weyl groups and much more.
If $\tilde{G}$ is also quasisplit over $k$,
then we can use this correspondence and our norm map $\map\NN{\tilde T}{T}$ above
to define a \emph{conorm map}
$\map{\dnorm[T^*]}{T^*}{\tilde T^*}$
for dual maximal $k$-tori $T^* \subseteq G^*$ and $\tilde T^* \subseteq\tilde G^*$,
and can obtain explicit embeddings of Weyl groups
$W(G^*,T^*) \longrightarrow W(\tilde{G}^*,\tilde{T}^*)$
(\S\ref{sec:conorm}).
In particular, we show (Proposition \ref{prop:centralizer-weyl})
that such embeddings have good restriction properties
with respect to centralizer subgroups of $G^*$.
We then have all of the ingredients in place to prove Statement \eqref{stmt:geom-conj-lift}
in \S\ref{sec:geometric-general}.
Using our cohomological results in \S\ref{sec:stability},
we can then prove Statement \eqref{stmt:stable-lift}
in \S\ref{sec:main},
and it is a simple matter to observe that Statement \eqref{stmt:rational-lift}
is just a special case.

In a future work,
we will address the problem of lifting other pieces of the parametrization
of irreducible representations of finite reductive groups,
such as unipotent conjugacy classes in dual groups.

\begin{acknowledgements}
We have benefited from conversations with
Jeffrey Adams, Brian Conrad, Stephen DeBacker, Jeffrey Hakim, Robert Kottwitz,
and Jiu-Kang Yu.
Thanks also to Bas Edixhoven for pointing out an error in an earlier version
of \S\ref{sec:group-actions},
to Brian Conrad for advice on fixing it,
to Avner Ash for coining the term ``parascopy'',
and to an anonymous referee for suggesting several improvements.
\end{acknowledgements}

\section{General notation and terminology}
\label{sec:setup}
Let $k$ denote a field,
and $k\sep$ denote the separable closure of $k$ in an algebraic closure $\bar k$ of $k$.
We will abbreviate $\Gal (k\sep/k)$ by $\Gal (k)$.
Given a connected reductive $k$-group $G$ and a maximal torus 
$T$ of $G$, let $\Phi(G,T)$ (resp.~$\Phi^\vee(G,T)$) denote the absolute root
(resp.~coroot) system of $G$ with
respect to $T$.
Let $W(G,T)$ denote the Weyl group of $G$ with respect to $T$.

For $g\in G(k\sep)$, let $\Int$ denote the natural homomorphism $\abmap{G}{{\Inn}(G)}$
given by $\Int (g)(x) =  gxg\inv$.
If $x\in G(k\sep)$ (resp.~$Y\subseteq G$), we will also denote $\Int(g)(x)$ 
(resp.~$\Int(g)(Y)$) by $\lsup g x$ (resp.~$\lsup g Y$).
For an algebraic $k$-group $G$,
let $G\conn$ denote its identity component.
A \emph{geometric conjugacy class} of $G$ is an orbit for the action of $G$ on itself via
conjugation.
Following Kottwitz~\cite{kottwitz:rational-conj},
we say that two elements $s_1,s_2 \in G(k)$
are \emph{stably conjugate} if there is some $g \in G(k\sep)$
such that $\lsup g s_1 = s_2$,
and for all $\sigma \in \Gal(k)$,
we have that $g\inv \sigma(g) \in C_G(s_1)\conn(k\sep)$.
Moreover, we say that two maximal $k$-tori $T_1,T_2\subseteq G$
are \emph{stably conjugate} in $G$
if there is some element $g\in G(k\sep)$
such that
$\Int(g)$ restricts to a $k$-isomorphism from $T_1$ to $T_2$.
Let $\TTst(G,k)$ denote the set of stable conjugacy classes of maximal $k$-tori in $G$.

If $\phi$ is a homomorphism from a group $\Gamma$ to the group of automorphisms of some object,
then
we will denote the operation of taking $\phi(\Gamma)$-fixed points by
$(\phantom{x})^{\phi(\Gamma)}$,
or just by
$(\phantom{x})^\Gamma$
when $\phi$ is understood.

For any $k\sep$-torus $T$,
let $\bX^*(T)$ and $\bX_*(T)$ respectively denote the
character and cocharacter modules of $T$.
Let $\langle\phantom{x},\phantom{x}\rangle$
denote the natural bilinear pairing between $\bX^*(T)$ and $\bX_*(T)$.
Let $V^*(T) = \bX^*(T)\otimes\Q$
and $V_*(T) = \bX_*(T)\otimes\Q$.
Then $\langle\phantom{x},\phantom{x}\rangle$ extends to a nondegenerate
pairing between the
$\Q$-vector spaces $V^*(T)$ and $V_*(T)$. 
Any homomorphism $\map f{T}{T'}$ of
tori determines  maps $\map {f^*}{\bX^*(T')}{\bX^*(T)}$
and $\map {f_*}{\bX_*(T)}{\bX_*(T')}$, and hence maps
$V^*(T')\longrightarrow V^*(T)$ and $V_*(T)\longrightarrow V_*(T')$
that we will also denote by $f^*$ and $f_*$, respectively.

For $i=1,2$, let $T_i$ be a maximal $k$-torus of $G$, and suppose that
$\lsup g T_1 = T_2$ for some $g\in G(k\sep)$.  
Then $\Int (g)$ gives an isomorphism
$T_1\rightarrow T_2$ (not necessarily defined over $k$).
For $\chi\in\bX^*(T_1)$ and $\lambda\in\bX_*(T_1)$,
define 
$$
\lsup g\chi :=\Int (g)^{*\, -1}\chi,\qquad \lsup g\lambda :=\Int (g)_*\lambda.
$$

For a root datum $(\bX^*,\Phi,\bX_*,\Phi^\vee)$ and a root $\alpha\in\Phi$, we denote by $\alpha^\vee$
the corresponding coroot in $\Phi^\vee$.

\section{Finite-group actions on character and cocharacter modules}
\label{sec:modules}
Let $\tilde \bX_*$ be a lattice of finite rank equipped with an
action of a finite group $\Gamma$. Then $\Gamma$ also acts on
the dual $\tilde \bX^*$ of $\tilde\bX_*$, as well as on
$\tilde V_* = \tilde\bX_*\otimes\Q$ and its dual $\tilde V^* = \tilde\bX^*\otimes\Q$.

Let $\bX_*$ be another lattice with dual $\bX^*$,
and let $V_* = \bX_*\otimes \Q$
and $V^* = \bX^*\otimes \Q$.
We will denote by $\langle\phantom{x},\phantom{x}\rangle$
the natural bilinear pairings between $V^*$ and $V_*$, and between $\tilde V^*$ and $\tilde V_*$.
Suppose that $\map{j_*}{V_*}{\tilde V_*^\Gamma}$ is an isomorphism
such that $j_*(\bX_*) \supseteq \tilde{\bX}_*^\Gamma$.
Composing $j_*$ with the natural inclusion $\abmap{\tilde V_*^\Gamma}{\tilde V_*}$,
and taking duals,
we obtain maps $i_*$, $j^*$, and $i^*$ as follows:
$$
\begin{xy}
\xymatrix{
V_*\ar@/^1.5pc/[rr]^{i_*} 
\ar@{->}[r]^{j_*}_{\sim}
& \tilde V_*^\Gamma
\ar@{^{ (}->}[r]
& \tilde V_*  
\ar@{.>>}@/^1.5pc/[ll]^{\pi}
}
\end{xy}
\qquad
\begin{xy}
\xymatrix{
V^*
\ar@{^{ (}.>}@/_1.5pc/[rr]_{\iota}
& \tilde V^*_\Gamma
\ar@{->}[l]_{j^*}^{\sim}
& \tilde V^*
\ar@{->>}[l]
\ar@/_1.5pc/[ll]_{i^*}
}
\end{xy}
$$
where $\iota$ and $\pi$ are to be described below.  Here $\tilde V^*_\Gamma$ denotes the space of
coinvariants for the action of $\Gamma$ on $\tilde V^*$.

Define $\map{\pi}{\tilde V_*}{V_*}$ to be
$$
j_*\inv \circ \biggl(\frac{1}{|\Gamma|}\sum_{\gamma\in\Gamma}\gamma\biggr).
$$
We have that 
$\pi \circ i_* = \mathrm{id}$,
so $\pi$ is a projection of $\tilde V_*$ onto $V_*$.

Let $\map{\iota}{V^*}{\tilde V^*}$ be the transpose of $\pi$.
More explicitly, one can show that if $\tilde v\in \tilde V^*$ is any
preimage under $i^*$ of $v$, then
\begin{equation}
\label{eq:iota}
\iota(v) = \frac{1}{|\Gamma|}\sum_{\gamma\in\Gamma}\gamma\cdot \tilde v.
\end{equation}

The image of $\iota$ is clearly the subspace $\tilde V^*{}^\Gamma$ of
$\Gamma$-fixed vectors in $\tilde V^*$.
Moreover, $\iota$
is injective (since $\pi$ is surjective)
and respects the bilinear pairings $\abmap{V^*\times V_*}{\Q}$ and
$\abmap{\tilde V^*\times \tilde V_*}{\Q}$
in the sense that for all $v\in V^*$ and $w\in V_*$, we have
$\langle \iota(v),i_*(w)\rangle = \langle v,w\rangle$.
Using $\iota$ (resp.~$j_*$), we may therefore identify $V^*$ (resp.~$V_*$) with
$\tilde V^*{}^\Gamma$ (resp.~$\tilde V_*^\Gamma$).
We note that 
\begin{equation}
\label{eq:restriction}
i^*  = \iota\inv \circ \biggl(\frac{1}{|\Gamma|}\sum_{\gamma\in\Gamma}\gamma\biggr) .
\end{equation}

Define a map $\map{\norm[*]}{\tilde V_*}{V_*}$ by
\begin{equation}
\label{eq:cochar-norm}
\normcochar = j_*\inv \circ \biggl(\sum_{\gamma\in\Gamma}\gamma\biggr) = |\Gamma|\,\pi .
\end{equation}
Taking duals, we obtain the adjoint map
$\map{\normchar}{V^*}{\tilde V^*}$.
Because of our assumption on $j_*$, we have that $\norm[*]$ takes $\tilde{\bX}_*$ to $\bX_*$,
and thus that $\normchar$ takes $\bX^*$ to $\tilde{\bX}^*$.

Suppose that $\tilde{\bX}_*$ and $\bX_*$ (and hence $\tilde{\bX}^*$ and $\bX^*$) are equipped with
$\Gal(k)$-actions.  Further, suppose that the action of $\Gal(k)$ on $\tilde{\bX}_*$ commutes with
that of $\Gamma$.  It follows that if $j_*$ is equivariant with respect to these actions
of $\Gal(k)$, then so are all of the other maps above.

\begin{example}
\label{ex:fixed-point}

Let $\tilde T$ be a $k$-torus and let $\Gamma$ be a finite group that acts on $\tilde{T}$ via
$k$-automorphisms.  Then $\Gamma$ acts on both
$\bX^*(\tilde T)$ and $\bX_*(\tilde T)$
(and thus on $\tilde V^* = V^*(\tilde T)$ and $\tilde V_*=V_*(\tilde T)$)
via the rules
$$
\gamma\cdot\chi = \gamma^{*\, -1}\chi ,\qquad
\gamma\cdot\lambda = \gamma_*\lambda
$$
for $\chi\in \bX^*(\tilde T)$ and $\lambda\in \bX_*(\tilde T)$.

Let $T$ be the $k$-torus $(\tilde T^\Gamma)\conn$.
Then we have the inclusion map $\map{i}{T}{\tilde{T}}$ and 
a norm map
$\map{\norm[T] = \norm}{\tilde{T}}{T}$
given by 
\begin{equation}
\label{eq:norm}
\norm(t) = \prod_{\gamma\in\Gamma} \gamma(t),
\end{equation}
both of which are defined over $k$.
Let $V^* = V^*(T)$ and $V_*=V_*(T)$.
The map $\map{i}{T}{\tilde{T}}$
induces an inclusion
$\map{i_*}{V_*}{\tilde V_*}$, which gives rise to an isomorphism
$\map{j_*}{V_*}{\tilde V_*{}^\Gamma}$. Then the maps $i_*$ and $i^*$ (resp.~$\norm[*]$ and $\normchar$)
constructed from $j_*$ in this section coincide with those induced by $i$ and $\norm$.
They are all $\Gal(k)$-equivariant.
\end{example}

\section{Finite-group actions on reductive groups}
\label{sec:group-actions}
In this section, we establish a strong form of Statement \eqref{stmt:G-reductive}
from the Introduction.

\begin{defn}
\label{defn:qss}
We say that an automorphism $\gamma$ of a connected reductive algebraic group $\tilde G$ is
\emph{quasi-semisimple} if $\gamma$ preserves a Borel subgroup $\tilde B_\bullet$ of $\tilde{G}$
and a maximal torus $\tilde T_\bullet$ in $\tilde B_\bullet$.
\end{defn}

\begin{rem}
\label{rem:automatically-separable}
If both $\tilde G$ and $\gamma$ are defined over $k$,
then the groups $\tilde B_\bullet$
and $\tilde T_\bullet$ can be chosen to be defined over $k\sep$.
We prove this in almost all cases in Lemma \ref{lem:rational-torus}.
A general proof of Lemma \ref{lem:rational-torus}, as well as proofs of
Lemmas~\ref{lem:steinberg} and~\ref{lem:rational},
can be found in a preprint of Lemaire
\cite{lemaire:twisted-characters}*{\S4.6}.
Our proofs are different.
\end{rem}

\begin{lem}
\label{lem:steinberg}
Suppose $\tilde{G}$ is a connected reductive algebraic group,
and $\gamma$ is a quasi-semisimple automorphism of $\tilde G$.
Let $\tilde{B}_\bullet$ be a $\gamma$-invariant Borel subgroup of $\tilde{G}$,
and $\tilde{T}_\bullet\subseteq \tilde{B}_\bullet$ a $\gamma$-invariant maximal torus of $\tilde{G}$.
Then the group 
$G := (\tilde{G}^\gamma)\conn$ is reductive,
$T_\bullet:= (\tilde T_\bullet^\gamma)\conn = G\cap\tilde T_\bullet$
is a maximal torus in $G$,
and $B_\bullet := (\tilde B_\bullet^\gamma)\conn$ is a Borel subgroup of $G$ containing $T_\bullet$.
\end{lem}

\begin{proof}
Let $\tilde G'$ be the simply connected cover of the derived group of $\tilde G$, and let $\tilde Z$ be the
identity component of the center of $\tilde G$.  Then we have a central isogeny
$\map\phi{\tilde{Z}\times\tilde{G}'}{\tilde{G}}$.  By~\cite{steinberg:endomorphisms}*{\S9.16},
the restriction of $\gamma$ to the derived group of $\tilde G$
lifts uniquely to an automorphism of $\tilde G'$, and thus
$\gamma$ lifts uniquely to an automorphism of $\tilde{Z}\times\tilde{G}'$.
Moreover, $\tilde{G}'$ contains a maximal torus $\tilde{T}'_\bullet$
and Borel subgroup $\tilde{B}'_\bullet$ such that
$\tilde{Z} \times \tilde{T}'_\bullet$ and
$\tilde{Z} \times \tilde{B}'_\bullet$ 
are the inverse images under $\phi$ of $\tilde{T}_\bullet$ and $\tilde{B}_\bullet$.
Note that $\gamma$ preserves $\tilde{T}'_\bullet$
and $\tilde{B}'_\bullet$.
From Steinberg \cite{steinberg:endomorphisms}*{Theorem 8.2},
we know that $G' := \tilde{G}'^\gamma$ is a connected reductive group, and it
is clear [\loccit, Remark 8.3(a)]
that $T'_\bullet:= \tilde T'_\bullet{}^\gamma = G'\cap\tilde T'_\bullet$
is a maximal torus in $G'$
and $B'_\bullet := \tilde B'_\bullet{}^\gamma$ is a Borel subgroup of $G'$ containing $T'_\bullet$.
We obtain our desired result by
considering
$\phi( (\tilde Z^\gamma)\conn \times T'_\bullet )$
and
$\phi( (\tilde Z^\gamma)\conn \times B'_\bullet )$.
\end{proof}

\begin{lem}
\label{lem:rational}
Suppose $\tilde{G}$ is a connected reductive $k$-group,
and $\gamma$ is a quasi-semisimple $k$-automorphism of $\tilde G$.
Then the group 
$G := (\tilde{G}^\gamma)\conn$ is defined over $k$.
\end{lem}
\begin{proof}
Let $\tilde T_\bullet$ and $\tilde B_\bullet$ be as in Definition~\ref{defn:qss}.
From Remark \ref{rem:automatically-separable},
we may assume that these groups are
defined over $k\sep$.
Since $\gamma$ is defined over $k$, we have that $\Gal(k)$ preserves $G$.
Therefore,
it will be enough to show that $G$ is defined over $k\sep$ (see~\cite{springer:lag}*{Prop.~11.2.8(i)}).
We may therefore
assume that $k=k\sep$.

Let $T_\bullet$ and $B_\bullet$ be as in Lemma \ref{lem:steinberg}.
Consider
the free part of the module of $\gamma$-coinvariants
in the lattice $\bX^*(\tilde T_\bullet)$.
Since this is a lattice, and it is a quotient of $\bX^*(\tilde T_\bullet)$,
it is the character lattice of a $k$-subtorus of $\tilde T_\bullet$,
and this subtorus is precisely $T_\bullet$.
That is, $T_\bullet$ is defined over $k$.
Since $\tilde T_\bullet$ is split over $k$, the root groups $U_\alpha$
for each $\alpha\in\Phi(\tilde G,\tilde T_\bullet)$
are defined over $k$,
and
there exist $k$-isomorphisms $x_\alpha$ from the 
additive group to each $U_\alpha$.
Moreover, if $\alpha$ is a root whose $\gamma$-orbit has size $r$, then
since $\gamma$ is defined over $k$, we may select the automorphisms
$x_{\gamma^i\alpha}$ so that $\gamma^i\circ x_\alpha = x_{\gamma^i\alpha}$
for $i = 0, \dots ,r-1$.  It follows from the 
description of the root groups for $G$ with respect to $T_\bullet$ given, say,
in part $(2'''')$
of the proof of \cite{steinberg:endomorphisms}*{Theorem 8.2},
that each such root group $U_\beta$ ($\beta\in\Phi(G,T_\bullet)$)
inherits a $k$-structure from that of $\tilde G$.
Thus, the product $Y = T_\bullet \prod U_\beta$ of $T_\bullet$ with all of the
root groups $U_\beta$ (in any order) is an open $k$-variety in $G$.
As a result,
$Y(k) \subset G(\bar k)\cap \tilde G(k)$ is dense in $Y$ by~\cite{springer:lag}*{Theorem~11.2.7},
and hence dense in $G$.  It follows
[\loccit, Lemma~11.2.4(ii)] that $G$ is defined over $k$.
\end{proof}

Note that if $\gamma$ is not quasi-semisimple, then $G$ need not be reductive.  For example,
suppose
$\tilde{G} = \GL(2)$, $g\in \tilde{G}(k)$ is a nontrivial unipotent element of finite order, and
$\gamma = \Int(g)$.

We now restrict our attention to automorphisms of finite order.
\begin{prop}
\label{prop:arb-gp-action-well-posed}
Suppose $\tilde{G}$ is a connected reductive $k$-group,
and $\Gamma$ is a finite group
that acts on $\tilde{G}$ via $k$-automorphisms
that preserve
a common Borel subgroup $\tilde B_\bullet$ of $\tilde{G}$
and maximal torus $\tilde T_\bullet$ in $\tilde B_\bullet$.
Let $G = (\tilde{G}^\Gamma)\conn$.
Then:
\begin{enumerate}[(i)]
\item
\label{item:G-reductive}
$G$ is a reductive $k$-group.
\item
\label{item:BT-Borus}
For every Borel-torus pair
$(\tilde{B},\tilde{T})$ in $\tilde{G}$ preserved by
$\Gamma$, we have that
$((\tilde{B}^\Gamma)\conn, (\tilde{T}^\Gamma)\conn)$
is a Borel-torus pair for $G$.
\item
\label{item:Ttilde-torus}
Let $T$ be a maximal torus in $G$, and let
$\tilde{T} = C_{\tilde{G}}(T)$.
Then $\tilde{T}$ is a maximal torus
in $\tilde{G}$.  
\item
\label{item:roots}
Let $T$ and $\tilde{T}$ be as in (\ref{item:Ttilde-torus}).  Then 
each root in $\Phi(G,T)$ is the restriction
to $T$ of a root in $\Phi(\tilde G,\tilde T)$.
\item
\label{item:Borel}
Let $\tilde{T}$ be as in (\ref{item:Ttilde-torus}).
Then there is some Borel subgroup $\tilde{B}$ of $\tilde{G}$
containing $\tilde{T}$ such that
$(\tilde{B},\tilde{T})$ is a Borel-torus pair preserved by $\Gamma$.
\end{enumerate}
\end{prop}

\begin{rems}\ 
\label{rems:big-thm}
\begin{enumerate}[(i)]
\item
The reductivity of $G$ is proved by Prasad and Yu \cite{prasad-yu:actions}
under somewhat different hypotheses.
Rather than assuming that $\Gamma$ fixes a Borel-torus pair, they assume that $|\Gamma|$
is not divisible by $\chr k$.
When $\chr k > 0$, their hypotheses
are neither stronger nor weaker than ours.

\item
\label{item:tildeToverk}
One can choose $T$ in part
(\ref{item:Ttilde-torus}) to be defined over $k$,
in which case the torus $\tilde{T}$ is also defined over $k$.
\end{enumerate}
\end{rems}

\begin{proof}
As in the proof of Lemma~\ref{lem:rational}, if $G$ is defined over $k\sep$, then it is defined over $k$.
Therefore, it will be enough to prove this result under the assumption that
$k=k\sep$.

Let $\tilde{Z}$ denote the identity component
of the center of $\tilde{G}$,
and $\tilde{G}'$ the derived subgroup of $\tilde{G}$.
Then we obtain an action of $\Gamma$ on $\tilde{Z}\times\tilde{G}'$
via $k$-automorphisms.
Let $C$ be the kernel of the central $k$-isogeny $\abmap{\tilde{Z}\times\tilde{G}'}{\tilde{G}}$.
Then the image of $((\tilde{Z}\times\tilde{G}')^\Gamma) \conn$ under the isogeny
is a normal subgroup of $\tilde{G}^\Gamma$; that its index is finite follows from the finiteness
of $H^1(\Gamma, C(\bar k))$ (see~\cite{weibel:homological-algebra}*{Cor.~6.5.10}). It follows that this image
is precisely $G$.
Thus, it is enough to prove the proposition with $\tilde{Z}\times\tilde{G}'$ in place of $\tilde{G}$.
Since the proposition is clear for $\tilde{Z}$, we may replace $\tilde{G}$ by $\tilde{G}'$.
That is, we may and do assume that $\tilde{G}$ is semisimple. Moreover, we note that the action of
$\Gamma$ on $\tilde G$ can be lifted to an action on the simply connected
cover of $\tilde G$ that satisfies all of the hypotheses of this theorem.
It follows from reasoning similar to that
above that we may further assume that $\tilde G$ is simply connected.
In particular, $\tilde{G}$ is a direct product of almost-simple groups, and $\Gamma$ permutes these.

Let $(\tilde{B}_\bullet,\tilde{T}_\bullet)$ denote a $\Gamma$-invariant
Borel-torus pair in $\tilde{G}$.

We first prove (\ref{item:Ttilde-torus}) and (\ref{item:Borel}).
Let $T_\bullet $ be the torus $(\tilde{T}_\bullet^\Gamma)\conn$.
We first show that
$T_\bullet$ is maximal in $G$
and
$C_{\tilde{G}}(T_\bullet) = \tilde{T}_\bullet$.
We prove the latter by showing that the root system 
$\Phi(C_{\tilde{G}}(T_\bullet),\tilde{T}_\bullet)$
is empty.
Let $\Phi^+$ denote the positive subsystem of $\Phi(C_{\tilde{G}}(T_\bullet),\tilde{T}_\bullet)$
determined
by $\tilde{B}_\bullet$.
Suppose for a contradiction that $\Phi^+$ is nonempty, and
choose $\alpha \in \Phi^+$.
Let $\chi = \sum_{\gamma\in \Gamma} \gamma\cdot\alpha$.
Since $\Gamma$ preserves $\tilde{B}_\bullet$, we have that
$\Gamma$ preserves $\Phi^+$, so $\chi$ is a positive linear combination
of elements of $\Phi^+$, and is thus nonzero.

The canonical pairing $\langle\phantom{x},\phantom{x}\rangle$
between $V^*(\tilde T_\bullet)$ and $V_*(\tilde T_\bullet)$ is invariant under the action
of $\Gamma$.
Thus, for all
$\lambda \in \bX_*(T_\bullet) = \bX_*(\tilde{T}_\bullet)^\Gamma$,
$\langle\chi,\lambda\rangle = \sum_{\gamma} \langle\gamma\cdot\alpha,\lambda\rangle
= |\Gamma| \langle\alpha,\lambda\rangle$, which is $0$ since $\alpha$ is a
root of the centralizer of $T_\bullet$.
Since $\chi$ is clearly $\Gamma$-invariant, we may identify it with a vector 
$\chi_0\in V^*(T_\bullet)$
via the map $\iota\inv$
as discussed in~\S\ref{sec:modules}.
(In fact, 
$\chi_0 = |\Gamma|\,i^*\alpha$ by (\ref{eq:restriction}).)
Then $\langle\chi_0,\lambda\rangle
= \langle\chi,\lambda\rangle=0$ for all $\lambda\in\bX_*(T_\bullet)$.  Since
$\langle\phantom{x},\phantom{x}\rangle$ is nondegenerate, it follows that
$\chi_0 = 0$ and hence $\chi = 0$,
a contradiction.

To see that $T_\bullet$ is maximal in $G$, consider a maximal torus $T'_\bullet$ 
containing $T_\bullet$.
Then
$$
T'_\bullet \subseteq C_{\tilde{G}}(T'_\bullet) \subseteq C_{\tilde{G}}(T_\bullet) = \tilde{T}.
$$
Taking identity components of groups of $\Gamma$-fixed points, we see that
$T'_\bullet \subseteq T_\bullet$, and thus the two tori are equal.

Now let $T$ denote an arbitary maximal torus in $G$.
We can write $T = \lsup g T_\bullet$ for some $g\in G(\bar k)$.
Thus,
$C_{\tilde{G}}(T) = \lsup g \tilde{T}_\bullet$, which is a maximal
torus in $\tilde{G}$.  Also, $\lsup g \tilde{B}_\bullet$ is a $\Gamma$-invariant
Borel subgroup of $\tilde G$ containing $\tilde T$.
This proves (\ref{item:Ttilde-torus}) and (\ref{item:Borel}).

We now prove
the other statements of the theorem
simultaneously by a series of reductions.

Suppose $\tilde{G} = \tilde{G}_1 \times \tilde{G}_2$,
a direct product of connected reductive $k$-groups, each
preserved by $\Gamma$.
By induction on the dimension of $\tilde{G}$,
we see that each statement holds for each $\tilde{G}_i$,
and thus for $\tilde{G}$.
It thus only remains to consider the case where $\tilde{G}$ has no
such decomposition.
Recall that we may assume that $\tilde{G}$ is simply connected.
In particular, we may assume that $\tilde{G} = \prod_{i=1}^r \tilde{G}_i$,
a direct product of almost simple $k$-groups which are permuted
transitively by $\Gamma$.

Let $\Gamma_1 = \stab_\Gamma(\tilde{G}_1)$.
For $1\leq i \leq r$, choose $\gamma_i\in \Gamma$ such that
$\gamma_i(\tilde G_1) = \tilde G_i$.
Let $\tilde{H} = \prod_{i=1}^r \tilde{G}_1$.
Then we have a $k$-isomorphism
$\map{\vec\gamma = (\gamma_1,\ldots,\gamma_r)}{\tilde{H}}{\tilde{G}}$.
Since $\Gamma$ is generated by $\Gamma_1$ and $\{\gamma_1,\ldots,\gamma_r\}$,
we have that
$$
\tilde{G}^\Gamma = \vec\gamma \bigl(\diag(\tilde{G}_1^{\Gamma_1})\bigr)
= \vec\gamma\bigl((\tilde{H})^{S_r \times \Gamma_1}\bigr),
$$
where $\map{\diag}{\tilde{G}_1}{\tilde{H}}$ is the diagonal embedding,
and $S_r$ is the symmetric group acting on $\tilde{H}$ by permutation of coordinates.
Thus, we may replace $\tilde{G}$ by $\tilde{H}$ and $\Gamma$ by $S_r\times \Gamma_1$.

Using induction on $|\Gamma|$ and working in stages,
we see that we may always
replace $\Gamma$ with the successive subquotients that occur
in any subnormal series for $\Gamma$.
Thus, we may assume that $\Gamma$ is simple.
Furthermore, from the previous paragraph,
we may assume either 
$\Gamma$ either acts via permutations of coordinates,
or 
preserves the simple factors of $\tilde{G}$ and acts in the same way on each.
In the former case, the proposition is clear.
Thus, we assume that we are in the latter case,
and we may assume from above that $\tilde{G}$ is almost simple.

Since $\Gamma$ is simple,
it either consists of inner automorphisms
(in which case it embeds in an isogenous image of $\tilde{T}_\bullet$),
or it embeds in the symmetry group of the Dynkin diagram of $\tilde{G}$.
In each case, $\Gamma$ is solvable, and thus cyclic.
Thus,
(\ref{item:G-reductive})
and
(\ref{item:BT-Borus})
follow from
Lemma~\ref{lem:steinberg} and Lemma~\ref{lem:rational}, and (\ref{item:roots})
follows from
the description of the root groups of $G$
in~\cite{steinberg:endomorphisms}*{\S8.2(2$''''$)}.
\end{proof}

\begin{defn}
If $\tilde G$ and $\Gamma$ are as in Proposition~\ref{prop:arb-gp-action-well-posed}, we
say the action of $\Gamma$ on $\tilde G$ is \emph{quasi-semisimple}.
\end{defn}

\section{Parascopy: Definition}
\label{sec:parascopy-defn}

We now axiomatize the essential properties of the relationship between the
group $G$ and the pair $(\tilde{G},\Gamma)$ of \S\ref{sec:group-actions}.

\begin{defn}
\label{defn:parascopy}
Let
$\Psi = (\bX^*, \Phi,\bX_*,\Phi^\vee)$
and
$\tilde\Psi = (\tilde\bX^*, \tilde\Phi,\tilde\bX_*,\tilde\Phi^\vee)$
be 
root data with $\Gal(k)$-actions, and $\Gamma$ a finite group.  Let $V^* = \bX^*\otimes\Q$,
$V_* = \bX_*\otimes\Q$, $\tilde V^* = \tilde\bX^*\otimes\Q$, $\tilde V_* = \tilde\bX_*\otimes\Q$.
A \emph{parascopic datum} for the triple $(\tilde\Psi,\Gamma,\Psi)$ is a pair
$(\phi,j_*)$, where
\begin{itemize}
\item
$\phi$ is a homomorphism from $\Gamma$ to $\Aut (\tilde\Psi)$,
such that $\phi(\Gamma)$ commutes with the action of $\Gal(k)$ and preserves some system of
positive roots in $\tilde\Phi$; and
\item
$\map{j_*}{V_*}{\tilde V_*^{\phi(\Gamma)}}$
is a $\Gal(k)$-equivariant isomorphism
\end{itemize}
satisfying the following two conditions.
\begin{condition}{P1}
$j_* (\bX_*) \supseteq \tilde\bX_*^{\phi(\Gamma)} $.
\end{condition}
Composing $j_*$ with the inclusion map $\abmap{\tilde V_*^{\phi(\Gamma)}}{\tilde V_*}$,
we obtain a map 
$\map{i_*}{V_*}{\tilde V_*}$
whose transpose
$\map{i^*}{\tilde V^*}{V^*}$ is assumed to satisfy:
\begin{condition}{P2}
$i^*\bigl(\tilde\Phi\bigr) \supseteq \Phi$.
\end{condition}
Let $G$ and $\tilde{G}$ be connected reductive $k$-groups, and $\Gamma$ a finite group.
Let $T$ (resp.~$\tilde T$) be a maximal $k$-torus of $G$ (resp.~$\tilde G$).  Let $\Psi(G,T)$
(resp.~$\Psi(\tilde G,\tilde T)$)
be the root datum of $G$ (resp.~$\tilde G$) relative to $T$ (resp.~$\tilde T$).
These root data come equipped with an action of $\Gal(k)$.
We will refer to a parascopic datum $(\phi,j_*)$ for the triple $(\Psi(\tilde G,\tilde T),\Gamma,\Psi(G,T))$
as a \emph{parascopic datum for $(\tilde{G},\Gamma,G)$
relative to the tori $\tilde T\subseteq \tilde{G}$ and $T\subseteq G$}.
We will say that $G$ is
a \emph{weakly parascopic group} for the pair $(\tilde{G},\Gamma)$
if such a parascopic datum exists, and we will feel free not to specify a particular datum if
it is clear from the context.
\end{defn}

\begin{examples}
\label{ex:weak-parascopy}
\ 
\begin{enumerate}[(a)]
\item
\label{item:group-action}
Suppose
$\Gamma$ acts quasi-semisimply on $\tilde{G}$,
and $G = (\tilde{G}^\Gamma)\conn$, as in \S\ref{sec:group-actions}.
From Proposition \ref{prop:arb-gp-action-well-posed} and Remark \ref{rems:big-thm}(\ref{item:tildeToverk}),
we can choose maximal $k$-tori $T\subseteq G$ and $\tilde{T}\subseteq \tilde{G}$
such that
$\Gamma$ preserves $\tilde{T}$, and
$T = (\tilde{T}^\Gamma)\conn$.
Then the given action $\map{\phi}{\Gamma}{\Aut_k(\tilde{T})}$,
together with the map $\map{j_*}{V_*(T)}{V_*(\tilde T)}$ induced by the inclusion
$\abmap{T}{\tilde{T}}$,
form a parascopic datum.
\item
\label{item:levi}
If $G$ is a Levi subgroup of $\tilde{G}$, then $G$ is weakly parascopic
for $(\tilde{G},1)$
with respect to an obvious parascopic datum.
\item
\label{item:isogenous}
If $G$ is the image under a central
$k$-isogeny of a weakly parascopic group for $(\tilde{G},\Gamma)$,
then $G$ is itself weakly parascopic
for $(\tilde{G},\Gamma)$.
\item
Our definition does not refer to any action of $\Gamma$ on $\tilde{G}$,
but
if $\tilde{G}$ is $k$-quasisplit
then one can indeed lift $\phi$ to a map $\abmap{\Gamma}{\Aut_k(\tilde{G})}$.
Different choices of lifting can lead to groups of fixed points $\tilde{G}^{\phi(\Gamma)}$
whose identity components are non-isomorphic.
We will show elsewhere
\cite{adler-lansky:lifting2}
that a particular lifting $\map{\phi_0}{\Gamma}{\Aut_k(\tilde{G})}$, one that fixes a pinning,
has the following
property.
If $\phi$ denotes any lifting, then
$(G^{\phi(\Gamma)})\conn$ is weakly parascopic
for
$((G^{\phi_0(\Gamma)})\conn, 1)$ in a natural way.
For example, considering two actions of $\Z/2\Z$ on $\GL(2n)$,
we will see that $\mathrm{SO}(2n)$ is parascopic for $(\mathrm{Sp}(2n),1)$.
\end{enumerate}
\end{examples}

In \S\ref{sec:parascopy-properties},
we will define and study a natural notion of equivalence between parascopic data.

\section{Finite-group actions and Weyl groups}
\label{sec:weyl-embedding}
Suppose that $\tilde\Psi = (\tilde\bX^*, \tilde\Phi,\tilde\bX_*,\tilde\Phi^\vee)$
and $\Psi = (\bX^*, \Phi,\bX_*,\Phi^\vee)$ are root data with $\Gal(k)$-action, and $\Gamma$ is a finite group.
Let $(\phi,j_*)$ be a parascopic datum for $(\tilde\Psi,\Gamma, \Psi)$.
As indicated in~\S\ref{sec:setup}, to ease notation, we will suppress reference to $\phi$ when
considering the action of $\Gamma$.
From now on, use the map $j_*$ to to identify
$V_*$ with $\tilde V_*^\Gamma$,
and the map $\iota$ of \S\ref{sec:modules} to identify $V^*$ with $(\tilde V^*)^\Gamma$.
Since $\Gamma$ preserves the root system $\tilde\Phi$, it acts on the Weyl group $\tilde W$ of $\tilde\Psi$.

In the particular situation where $\Gamma$ acts on $\tilde G$
and $G = (\tilde G^\Gamma)\conn$,
 there is an obvious embedding
$W(G,T)\longrightarrow W(\tilde G,\tilde T)^\Gamma$.  Here $w\in W(G,T)$ corresponds to the unique
element of $W(\tilde G,\tilde T)^\Gamma$ whose action when restricted to $T$ coincides
with that of $w$.  We now show that we have such an embedding in the more general setting
of parascopy.

Let $\alpha$ be a root in $\Phi$.
Then $\alpha = i^*\tilde\alpha$ for some root $\tilde\alpha$
in $\tilde\Phi$ by Condition P2 in the definition of parascopic datum.
There are two cases
to consider:
\begin{enumerate}
\item The roots in $\Gamma\cdot\tilde\alpha$ are mutually orthogonal.
\item The roots in $\Gamma\cdot\tilde\alpha$ are not mutually orthogonal.
\end{enumerate}

In case (1), let $\Xi = \Gamma\cdot\tilde\alpha$.  

In case (2),
\cite{kottwitz-shelstad:twisted-endoscopy}*{\S1.3} implies that for each 
$\theta\in\Gamma\cdot\tilde\alpha$, there exists a unique root 
$\theta'\neq\theta$ in $\Gamma\cdot\tilde\alpha$ such that $\theta$ and $\theta'$ are
not orthogonal.  Moreover, $\theta+\theta'$ is a root in $\Phi$ and does
not belong to $\Gamma\cdot\tilde\alpha$.  
Let $\Xi = \{\theta+\theta'\mid\theta\in\Gamma\cdot\tilde\alpha\}$.

\begin{rem}
Although it is assumed in~\cite{kottwitz-shelstad:twisted-endoscopy}
that $\Gamma$ is cyclic, there is only one case in which the action of the stabilizer in a
general group $\Gamma$ of 
an irreducible component of $\Phi$ need not factor through a cyclic quotient
(recall that $\Gamma$ must preserve a positive system of roots):
namely, when the component is of type $D_4$.  One easily checks that in this situation, case (1) holds.
\end{rem}

\begin{rem}
\label{rem:psi}
We note that in both cases, $\Xi$ is an orbit of mutually orthogonal roots.
\end{rem}

\begin{lem}
\label{lem:multiple}
Suppose that 
$\alpha \in \Phi$.
Then
with $\tilde\alpha$ and $\Xi$ as above, we have
$$
\sum_{\beta \in\Xi} \beta\spcheck = \frac{|\Xi|}{|\Gamma \cdot \tilde\alpha|} \alpha\spcheck.
$$
\end{lem}

\begin{proof}
By \eqref{eq:iota}, we have
\begin{equation}
\label{eq:alpha}
\sum_{\beta\in\Xi}\beta = \sum_{\beta\in\Gamma\cdot\tilde\alpha}\beta
 = |\Gamma\cdot\tilde\alpha|\alpha.
\end{equation}
Since $\sum_\Xi\beta$ is a multiple of $\alpha$, and the roots in $\Xi$ all have the same length,
it follows that 
$\sum_\Xi\beta^\vee$ is a multiple of $\alpha^\vee$.
To determine this multiple we let $\beta_0 \in \Xi$, and compute
\begin{align*}
\Bigl\langle \alpha,\sum_{\beta\in\Xi} \beta^\vee\Bigr\rangle 
&= \frac{1}{|\Gamma\cdot\tilde\alpha|}
\Bigl\langle \sum_{\beta'\in\Xi}\beta',\sum_{\beta\in\Xi} 
\beta^\vee\Bigr\rangle
&& \text{(by \eqref{eq:alpha})} \\
&=  \frac{|\Xi|}{|\Gamma\cdot\tilde\alpha|}
\left\langle \beta_0, \beta_0^\vee\right\rangle
&& \text{(by Remark \ref{rem:psi})} \\
&= \frac{|\Xi|}{|\Gamma\cdot\tilde\alpha|}
\left\langle \alpha, \alpha^\vee\right\rangle .
\end{align*}
Our result follows.
\end{proof}

\begin{lem}
\label{lem:fidelite}
The natural action of $\tilde W^\Gamma$ on $(\tilde V^*)^\Gamma$ is faithful.
\end{lem}
When $\Gamma$ is cyclic, this is~\cite{steinberg:endomorphisms}*{\S1.32(a)}.  
\begin{proof}
Let $w$ be a nontrivial element of $\tilde W^\Gamma$.
Let $\tilde\Delta$ be a $\Gamma$-invariant set of simple roots in $\tilde\Phi$
(guaranteed to exist by Definition~\ref{defn:parascopy}). Then there exists $\tilde\alpha\in \tilde\Delta$
such that $w(\tilde\alpha)\in -\tilde\Delta$.
It follows that $w(\gamma\cdot\tilde\alpha) = \gamma\cdot (w \tilde\alpha)$ belongs to $-\tilde\Delta$
for every $\gamma\in\Gamma$.
Let
$v = \sum\gamma\cdot\tilde\alpha\in \tilde V^*{}^\Gamma$, where
the sum runs over all $\gamma\in\Gamma$.
Then $w(v)$ is a linear combination of roots in $\tilde\Delta$ in which all
of the coefficients are nonpositive.  In particular, $w(v)\neq v$.
\end{proof}

Let $W$ denote the Weyl group of $\Psi$.

\begin{prop}
\label{prop:weyl-embedding}
There is a natural $\Gal(k)$-equivariant embedding
$W\longrightarrow \tilde W^\Gamma$.  Under this map,
the image of the reflection $w_\alpha$ through the root $\alpha\in\Phi$ is 
\begin{equation}
\label{eq:weyl-lift}
\tilde w = \prod_{\beta\in\Xi}w_\beta,
\end{equation}
where $\Xi$ is as above and the product is taken in any order.
\end{prop}
\begin{proof}
We may identify
$\Aut_\Q (V^*)$
with 
$\Aut_\Q (\tilde V^*{}^\Gamma)$,
and thus identify 
$W$ with a subgroup of the latter.
By Lemma~\ref{lem:fidelite}, there is a natural injection
$$
\tilde W^\Gamma\longrightarrow\Aut_\Q (\tilde V^*{}^\Gamma).
$$
To construct an embedding $W\longrightarrow \tilde W^\Gamma$, it is therefore
enough to show that the image of this injection contains $W$.
Thus, given $w\in W$, we will show that there exists $\tilde w \in \tilde W^\Gamma$
whose action on $\tilde V^*{}^\Gamma$ coincides with that of $w$.
It suffices to prove the existence of $\tilde w$ only in the case in which $w$ is a reflection
$w_\alpha$ through a root $\alpha\in\Phi$.  In this case, our candidate for $\tilde w$ is given by 
(\ref{eq:weyl-lift}).

Let $v\in V^* = \tilde V^*{}^\Gamma$.  Since the roots in the orbit $\Xi$ are
orthogonal, we have
\begin{equation}
\label{eq:tilde_w}
\tilde w(v) = v - \sum_{\beta\in\Xi}\langle v,\beta^\vee\rangle\beta .
\end{equation}
Since the pairing $\langle\phantom{x},\phantom{x}\rangle$ is $\Gamma$-invariant and
$v$ is $\Gamma$-fixed, the right-hand side of (\ref{eq:tilde_w}) is equal to
\begin{equation}
\label{eq:tilde_w_2}
v - \langle v,\beta_0^\vee\rangle \sum_{\beta\in\Xi}\beta .
\end{equation}
for any root $\beta_0\in\Xi$.
From \eqref{eq:alpha},
\begin{align*}
\langle v,\beta_0^\vee\rangle \sum_{\beta\in\Xi}\beta
 &= |\Gamma\cdot\tilde\alpha| \langle v,\beta_0^\vee\rangle\alpha\nonumber\\
 &= \frac{|\Gamma\cdot\tilde\alpha|}{|\Xi|}
\Bigl\langle v, \sum_{\beta\in\Xi}\beta^\vee\Bigr\rangle
\alpha  \\
&= \langle v,\alpha^\vee\rangle\alpha
&& \text{(by Lemma \ref{lem:multiple})}.
\end{align*}
Hence
by (\ref{eq:tilde_w}) and (\ref{eq:tilde_w_2}), we have that $\tilde w(v) = w_\alpha(v)$.
It follows that $\tilde w_\alpha = \tilde w$.

The $\Gal(k)$-equivariance of this embedding follows from the explicit formula~\eqref{eq:weyl-lift}.
\end{proof}

\section{Stable conjugacy classes of maximal $k$-tori}
\label{sec:stability}

The statement and proof of
the next result
are essentially the same as 
those for
\cite{reeder:elliptic-centralizers}*{Proposition 6.1}.
We include a proof here only because our hypotheses and conclusion are slightly different.
Similar proofs can be found in 
\cite{howard:thesis}*{\S2} and \cite{raghunathan:tori}.

\begin{prop}
\label{prop:stable-tori}
Let $G$ denote a connected reductive group over a field $k$.
Let $S$ be a maximal $k$-torus of $G$.
Then there is a natural injection $\abmap{\TTst(G,k)}{H^1(k,W(G,S))}$.
This map is a surjection when $G$ is $k$-quasisplit.
\end{prop}

\begin{proof}
Let $T$ be a maximal $k$-torus of $G$.
Then there exists $g\in G(k\sep)$ such that $\lsup g S = T$.
Let $f\in Z^1(k,N_G(S))$ be the the cocycle $\sigma\mapsto g\inv\sigma(g)$,
and let $\bar f$ be the image of $f$ in $Z^1(k,W(G,S))$.
Associate to $T$ the class of $\bar f$ in $H^1(k,W(G,S))$.
This class is independent of the particular element $g$.
The proof that this cohomology class uniquely determines and is determined by the stable conjugacy class of $T$
is a standard exercise (cf.~\cite{raghunathan:tori}*{\S1.1}).
Thus, we have the desired injection.

Suppose that $G$ is $k$-quasisplit.
Then $G$ contains a maximal $k$-torus $T_0$ that is the centralizer of
a maximal $k$-split torus.
Applying the above paragraph to the case of $S=T_0$,
we obtain an injection
$\abmap{\TTst(G,k)}{H^1(k,W(G,T_0))}$,
and this map is a surjection from~\cite{raghunathan:tori}*{Thm.~1.1}.
(Although this surjectivity result is stated only for semisimple groups,
it is easily extended to the reductive case.)

Now suppose $S$ is an arbitrary maximal $k$-torus of $G$.  There is some element $h\in G(k\sep)$ such that
$\lsup h T_0 = S$.  The map $\Int(h)$ then induces a natural bijection (of sets, though not of pointed sets)
$H^1(k,W(G,T_0)) \longrightarrow H^1(k,W(G,S))$ (cf.~\cite{serre:galois}*{Ch~I, Prop.~35}); one sees easily
that this bijection does not depend on the choice of $h$.
Composing the map from the preceding paragraph with this bijection yields the desired bijection.
\end{proof}

\section{Parascopy: Basic properties}
\label{sec:parascopy-properties}

We are now equipped to state and prove the basic properties of parascopy.
In this section,
let $\tilde{G}$ and $G$ denote connected reductive $k$-groups, $\Gamma$ a finite group,
and $\tilde{T}$ and $T$ maximal $k$-tori in $\tilde{G}$ and $G$.
Suppose that $(\phi,j_*)$ is a parascopic datum for $(\tilde{G},\Gamma,G)$ relative
to $\tilde{T}$ and $T$.
Thus, $G$ is weakly parascopic for $(\tilde{G},\Gamma)$.

\begin{defn}
\label{defn:parascopic-group}
Identify $\TTst(G,k)$ and $\TTst(\tilde{G},k)$ with subsets
of $H^1(k,W(G,T))$ and $H^1(k,W(\tilde{G},\tilde{T}))$, respectively, via the injection
of Proposition \ref{prop:stable-tori}. 
Consider the map $\abmap{H^1(k,W(G,T))}{H^1(k,W(\tilde{G},\tilde{T}))}$ induced
by the embedding given in Proposition \ref{prop:weyl-embedding}.
If this map restricts to give a map $\abmap{\TTst(G,k)}{\TTst(\tilde{G},k)}$,
then we will say that $G$ is a \emph{parascopic group} for $(\tilde{G},\Gamma)$.
\end{defn}

In other words, a weakly parascopic group $G$ is parascopic if every maximal $k$-torus in $G$
determines a unique one in $\tilde{G}$, up to stable conjugacy.

\begin{examples}
\label{ex:parascopic-group}
\ 
\begin{enumerate}[(a)]
\item
If $G$ is weakly parascopic for $(\tilde{G},\Gamma)$,
and $\tilde{G}$ is quasisplit over $k$, then
$G$ is parascopic by the surjectivity statement in Proposition \ref{prop:stable-tori}.
\item
In the situation of Proposition \ref{prop:arb-gp-action-well-posed},
$G$ is parascopic by part
(\ref{item:Ttilde-torus}) of this result.
\item
In Examples \ref{ex:weak-parascopy}(\ref{item:group-action},\ref{item:levi}),
$G$ is parascopic for $\tilde{G}$.
\end{enumerate}
\end{examples}

\begin{defn}
\label{defn:equivalent-data}
Let $(\phi',j_*')$ denote another parascopic datum for
$(\tilde G,\Gamma,G)$,
this time relative to the maximal $k$-tori $T'\subseteq G$ and $\tilde{T}'\subseteq \tilde{G}'$.
We say that the parascopic data $(\phi,j_*)$ and $(\phi',j_*')$ are
\emph{equivalent}
if there exist elements $g\in G(k\sep)$ and $\tilde{g}\in \tilde{G}(k\sep)$ satisfying:
\begin{enumerate}[(a)]
\item
$\lsup g T  = T'$ and $\lsup {\tilde{g}} \tilde{T} = \tilde{T}'$.
\item
\label{item:phi-compatible}
For all $\gamma\in \Gamma$,
$\phi'(\gamma) = \Int(\tilde{g})_* \circ \phi(\gamma) \circ \Int(\tilde{g})_*\inv$,
where $\phi(\gamma)$ and $\phi'(\gamma)$ are taken to be automorphisms of $\bX_*(\tilde{T})$
and $\bX_*(\tilde{T}')$.
\item
\label{item:j-compatible}
$j_*' = \Int(\tilde{g})_* \circ j_* \circ \Int(g)_* \inv$.
\end{enumerate}
In this case,
we will say that $(\phi,j_*)$ is equivalent to $(\phi',j_*')$ \emph{via the elements $g$ and $\tilde{g}$.}
\end{defn}

It is straightforward to verify that this is indeed an equivalence relation on the set of parascopic data
for $(\tilde{G},\Gamma,G)$.

Let $g$ and $\tilde g$ be as in Definition~\ref{defn:equivalent-data}. For $\sigma \in \Gal(k)$,
we have that $g\inv\sigma(g)\in N_G(T)(k\sep)$ and
$\tilde g\inv\sigma(\tilde g)\in N_{\tilde G}(\tilde T)(k\sep)$. In the following lemma and the remainder of
the paper, we will use bars to represent the natural maps from normalizers to Weyl groups of maximal tori.

\begin{lem}
\label{lem:equivalence}
Let $(\phi, j_*)$ be a parascopic datum for $(\tilde G,\Gamma,G)$ relative to $\tilde T$ and $T$.
Let $g\in G(k\sep)$ and $\tilde g\in\tilde G(k\sep)$, and suppose that the maximal tori
$T':=\lsup{g} T$ and 
$\tilde T':=\lsup{\tilde g}\tilde T$ are defined over $k$.
Define $\map{\phi'}{\Gamma}{\Aut(\bX_*(\tilde T'))}$ and $\map{j'_*}{V_*(T')}{V_*(\tilde T')}$ as in
Definition~\ref{defn:equivalent-data}(b,c).
Then $(\phi',j'_*)$ is a parascopic datum for $(\tilde G,\Gamma ,G)$
relative to $\tilde T'$ and $T'$
(necessarily equivalent to $(\phi,j_*)$) if and only if
$$
i\bigl(\overline {g\inv\sigma(g)}\bigr) = \overline{\tilde{g}\inv\sigma(\tilde{g})}
$$
for all $\sigma \in \Gal(k)$, 
where $\map{i}{W(G,T)}{W(\tilde{G},\tilde{T})^{\phi(\Gamma)}}$ is the 
embedding given in Proposition~\ref{prop:weyl-embedding}.
\end{lem}
\begin{proof}
Suppose $(\phi',j'_*)$ is a parascopic datum for $(\tilde G,\Gamma ,G)$.
Fix $\sigma\in\Gal(k)$, and let $w = \overline {g\inv\sigma(g)}\in W(G,T)$ and
$\tilde w = \overline{\tilde{g}\inv\sigma(\tilde{g})}\in W(\tilde{G},\tilde{T}) $.
Let $\gamma\in\Gamma$. Since $\phi(\gamma)$ and $\phi'(\gamma)$ are $\Gal(k)$-equivariant, we have
$$
\Int(\tilde{g})_* \circ \phi(\gamma) \circ \Int(\tilde{g})_*\inv = \phi'(\gamma) = 
\Int(\sigma(\tilde g))_* \circ \phi(\gamma) \circ \Int(\sigma(\tilde g))_*\inv .
$$
It follows that $\Int(\tilde{g}\inv\sigma(\tilde g))_*$ is $\phi(\Gamma)$-equivariant and hence that
$\tilde w\in W(\tilde{G},\tilde{T})^{\phi(\Gamma)}$. By Lemma~\ref{lem:fidelite},
to show that $i(w) = \tilde w$, it suffices to show that the actions of $w$ on $V_*(T)$ and of
$\tilde w$ on $V_*(\tilde T)$ correspond under the identification
$\map{j_*}{V_*(T)}{V_*(\tilde T)^{\phi(\Gamma)}}$, i.e.,
that
\begin{equation}
\label{eq:compatible}
j_*\circ w = \tilde w\circ j_*
\end{equation}
as maps
$\abmap{V_*(T)}{V_*(\tilde T)^{\phi(\Gamma)}}$.
But since $j_*$ and $j'_*$ are $\Gal(k)$-equivariant, we have
$$
\Int(\tilde{g})_* \circ j_* \circ \Int(g)_*\inv = j'_*
=\Int(\sigma(\tilde{g}))_* \circ j_* \circ \Int(\sigma(g))_*\inv ,
$$
and \eqref{eq:compatible} follows immediately.

The converse is proved similarly.
\end{proof}

The next result addresses the question:
how large is the equivalence class of the parascopic datum $(\phi,j_*)$?

\begin{prop}
\label{prop:parascopy-equivalence}
\begin{enumerate}[(i)]
\item
\label{item:all-tori-equally-good}
If $G$ is parascopic for $(\tilde{G},\Gamma)$ via the datum $(\phi,j_*)$, then
for every maximal $k$-torus $T'\subseteq G$,
and every element $g\in G(k\sep)$ such that $T' = \lsup{g} T$,
there exists a maximal $k$-torus $\tilde{T}'\subseteq \tilde{G}'$
and a parascopic datum $(\phi',j_*')$ for $(\tilde{G},\Gamma,G)$ relative to $\tilde{T}'$ and $T'$,
such that $(\phi',j_*')$ is 
equivalent to $(\phi,j_*)$ via $g$ and some $\tilde{g} \in \tilde{G}(k\sep)$.
\item
\label{item:Ttilde-stably-determined}
For a maximal $k$-torus $\tilde{T}'\subseteq \tilde{G}$,
there exists a parascopic datum equivalent to $(\phi,j_*)$
relative to 
$\tilde{T}'$ and $T$
if and only if $\tilde{T}'$ is stably conjugate to $\tilde{T}$ in $\tilde{G}$.
\item
\label{item:same-tori-equivalence}
Suppose $(\phi',j_*')$ is another parascopic datum for $(\tilde{G},\Gamma,G)$ relative to $\tilde{T}$ and $T$.
Then the two data are equivalent if and only if
there is an element
$\tilde w \in W(\tilde{G},\tilde{T})^{\Gal(k)}$
such that
$\phi'(\gamma) = \tilde w \circ \phi(\gamma) \circ \tilde w\inv$ for all $\gamma \in \Gamma$,
and
$j_*' =  \tilde w \circ  j_*$.
\end{enumerate}
\end{prop}

\begin{proof}
\begin{enumerate}[(i)]
\item
Since $G$ is parascopic for $(\tilde{G},\Gamma)$, we have a map
$\abmap{\TTst(G,k)}{\TTst(\tilde{G},k)}$ induced by the map
$\abmap{H^1(k,W(G,T))}{H^1(k,W(\tilde{G},\tilde{T})^{\phi(\Gamma)})}$.
Let $\tilde T'$ be a torus in the image of
the class of $T'$.
Pick $\tilde g\in \tilde G(k\sep)$
such that $\lsup{\tilde g} \tilde T = \tilde T'$.
The function $f$ (resp.~$\tilde f$) on $\Gal(k)$ given by $\mapto{f}{\sigma}{\overline{g\inv \sigma(g)}}$
(resp.~$\mapto{\tilde f}{\sigma}{\overline{\tilde g\inv \sigma(\tilde g)}}$) is a cocycle in
$Z^1(k,W(G,T))$ (resp.~$Z^1(k,W(\tilde{G},\tilde{T})^{\phi(\Gamma)})$).
By the definition of the map $\abmap{\TTst(G,k)}{\TTst(\tilde{G},k)}$, $\tilde f$
is cohomologous  to the image of
$f$ in $Z^1(k,W(\tilde{G},\tilde{T}))$. 
Moreover, by adjusting $\tilde g$ by an appropriate element of
$N_{\tilde G}(\tilde T)(k\sep)$, one can arrange for these cocycles to coincide.
In other words, $i\bigl(\overline {g\inv\sigma(g)}\bigr) = \overline{\tilde{g}\inv\sigma(\tilde{g})}$
for all $\sigma\in\Gal(k)$.
But then Lemma~\ref{lem:equivalence} implies that the pair $(\phi', j'_*)$ is a parascopic
datum for $(\tilde G,\Gamma,G)$ relative to $\tilde T'$ and $T'$,
and that it is equivalent to $(\phi,j_*)$.
\item
Suppose such a datum exists, equivalent to $(\phi,j_*)$ via $g\in N_G(T)(k\sep)$
and $\tilde g \in \tilde{G}(k\sep)$.  We may choose $\tilde n\in N_{\tilde G}(\tilde T)(k\sep)$
such that $\overline{\tilde n} = i(\bar g)\inv$.  Since
$i \bigl(\overline{g\inv \sigma(g)  }\bigr) = \overline{\tilde{g}\inv \sigma(\tilde{g})}$
by Lemma~\ref{lem:equivalence},
it follows that $\overline{(\tilde{g}\tilde{n})\inv \sigma(\tilde{g}\tilde{n})} = 1$
in $W(\tilde{G},\tilde{T})$ for all $\sigma\in\Gal(k)$.
Thus the map $\map{\Int(\tilde{g}\tilde{n})}{\tilde{T}}{\tilde{T}'}$ is defined over $k$,
and so $\tilde{T}$ and $\tilde{T}'$ are stably conjugate.
The converse is proved similarly.
\item
Suppose that the equivalence is via the elements $g\in G(k\sep)$ and $\tilde g \in \tilde{G}(k\sep)$.
Then $g$ and $\tilde g$ normalize $T$ and $\tilde{T}$, so we have
elements $w = \overline g$ and $\tilde w' = \overline{\tilde{g}}$ in $W(G,T)$ and $W(\tilde{G},\tilde{T})$.
Then Lemma~\ref{lem:equivalence} implies that
$\sigma ( \tilde w'\cdot i(w)\inv) = \tilde w' \cdot i(w)\inv$
for all $\sigma\in \Gal(k)$. Thus $\tilde w:=\tilde w' \cdot i(w)\inv\in W(\tilde{G},\tilde{T})^{\Gal(k)}$.
Moreover, since $i(w)\in W(\tilde{G},\tilde{T})^{\phi(\Gamma)}$,
it follows that for any $\gamma\in\Gamma$,
$$
\phi'(\gamma) = \tilde w'\circ\phi(\gamma)\circ \tilde w'{}\inv =
\tilde w\circ\phi(\gamma)\circ\tilde w\inv.
$$
By the definition of the embedding $i$,
$$
j'_* = \tilde w'\circ j_*\circ w\inv = (\tilde w'\cdot i(w)\inv)\circ j_* = \tilde w\circ j_*.
$$
The converse is proved similarly.
\qedhere
\end{enumerate}
\end{proof}

\section{Duality}
\label{sec:duality}
Let $G$ be a quasisplit connected reductive $k$-group.
Let $B_0$ denote a Borel $k$-subgroup of $G$,
and $T_0$ a maximal $k$-torus in $B_0$.
Suppose that $(G^*,B_0^*, T_0^*)$ is another triple of such groups.
We say that the two triples are in \emph{$k$-duality} if
there is a $\Gal(k)$-equivariant isomorphism
$\map{\delta_0}{\bX^*(T_0)}{\bX_*(T_0^*)}$ that induces an isomorphism
from the based root datum of $(G,B_0,T_0)$ to the dual of that of $(G^*,B_0^*,T_0^*)$;
that is,
\begin{itemize}
\item $\delta_0$ maps the simple roots in $\Phi(G,T_0)$ with respect to $B_0$ onto the 
simple coroots in $\Phi^\vee(G^*,T^*_0)$ with respect to $B^*_0$.
\item The transpose $\map{\delta_0^*}{\bX^*(T_0^*)}{\bX_*(T_0)}$ of $\delta_0$ maps
the simple roots in $\Phi(G^*,T_0^*)$ with respect to $B_0^*$ onto the 
simple coroots in $\Phi^\vee(G,T_0)$ with respect to $B_0$.
\end{itemize}
Given a triple $(G,B_0,T_0)$, such a triple $(G^*,B_0^*,T_0^*)$
always exists, and is unique up to $k$-isomorphism.
In this situation, we will say that $G^*$ is the \emph{$k$-dual}
of $G$.
We will say that
a pair of maximal $k$-tori
$T\subseteq G$ and $T^*\subseteq G^*$
are in \emph{$k$-duality}
if there is a $\Gal(k)$-equivariant isomorphism
$\map{\delta}{\bX^*(T)}{\bX_*(T^*)}$ such that
\begin{equation}
\label{cond:roots-to-coroots}
\delta (\Phi(G,T)) = \Phi^\vee(G^*,T^*)\quad\mbox{and}\quad 
\delta^* (\Phi(G^*,T^*)) = \Phi^\vee(G,T),
\end{equation}
where $\map{\delta^*}{\bX^*(T^*)}{\bX_*(T)}$ is the transpose of $\delta$.
(Note that this notion of duality of tori depends on the ambient groups
$G$ and $G^*$.)
The isomorphism $\delta$ will be referred to as a
\emph{duality map}.

\begin{rem}
\label{rem:weyl-isomorphism}
Note that a duality map $\delta$ determines a $\Gal(k)$-equivariant isomorphism, also denoted $\delta$,
from $W(G,T)$ to $W(G^*,T^*)$ under which, for each $\alpha\in\Phi(G,T)$,
the reflection $w_\alpha$ is sent to $w_{\delta\alpha}$.
Then for every $w\in W(G,T)$ and every $\chi\in \bX^*(T)$,
$\lsup w \chi = \lsup{\delta(w)} \delta(\chi)$.  In other words, if we identify $W(G,T)$ and $W(G^*,T^*)$
via $w_\alpha\longleftrightarrow w_{\delta\alpha}$, then 
$\map{\delta}{\bX^*(T)}{\bX_*(T^*)}$ is a $W(G,T)$-equivariant map.

We note that when $k$ is a finite field, it is standard
(see~\cites{deligne-lusztig:finite,carter:finite,digne-michel:reps-book}) to work
with an \textit{anti-action} of $W(G,T)$ on $\bX^*(T)$ (satisfying
$\lsup{w_1w_2}\chi = \lsup{w_2}(\lsup{w_1}\chi)$).  This, in turn, makes
the natural map $W(G,T)\longrightarrow W(G^*,T^*)$ an \textit{anti-isomorphism} and 
forces the geometric Frobenius element $F$ to act on the Weyl group $W(G^*,T^*)$
via the \textit{inverse} of its usual action on $W(G,T)$.
However, we consider the
standard action of $W(G,T)$ on $X^*(T)$, which results in the 
isomorphism $\delta$ of Weyl groups in the preceding paragraph.  Using this map to identify
$W(G,T)$ and $W(G^*,T^*)$, one sees easily that $\Gal(k)$ acts in the same way on
these groups.
\end{rem}

\begin{prop}
\label{prop:tori-duality}
There is a canonical one-to-one correspondence
$\TTst(G,k) \longleftrightarrow \TTst(G^*,k)$.
If $T\subseteq G$ and $T^*\subseteq G^*$ correspond, then they are in $k$-duality,
and the duality map $\map\delta{\bX^*(T)}{\bX_*(T^*)}$
is uniquely determined up to the action of $W(G,T)^{\Gal(k)}$.
\end{prop}

\begin{proof}
Fix maximal $k$-tori $T_0\subseteq G$ and $T_0^*\subseteq G^*$ and
a duality map $\map{\delta_0}{\bX^*(T_0)}{\bX_*(T_0^*)}$ of based root data as in
the definition of $k$-duality above.
As described in Proposition~\ref{prop:stable-tori}, we have a bijection between the
set of stable conjugacy classes of maximal $k$-tori of $G$ (resp.~$G^*$)
and $H^1(k,W(G,T_0))$ (resp.~$H^1(k,W(G^*,T^*_0))$).
Since $\bimap{\delta_0}{W(G,T_0)}{W(G^*,T_0^*)}$
induces an isomorphism
$H^1(k,W(G,T_0))\stackrel{\sim}{\longrightarrow} H^1(k,W(G^*,T_0^*))$
we have the desired correspondence between the sets of stable conjugacy classes
of maximal $k$-tori in $G$ and $G^*$.  

Let $T\subseteq G$ and $T^*\subseteq G^*$ be maximal $k$-tori whose stable
conjugacy classes correspond as in the preceding paragraph.
Then $T = \lsup g T_0$ and $T^* = \lsup {g^*} T^*_0$ for some $g\in G(k\sep)$ and $g^*\in G^*(k\sep)$.
The function $f$ (resp.~$f^*$) on $\Gal(k)$ given by $\mapto{f}{\sigma}{\overline{g\inv \sigma(g)}}$
(resp.~$\mapto{f^*}{\sigma}{\overline{\tilde {g^*}\inv \sigma(\tilde g^*)}}$) is a cocycle in
$Z^1(k,W(G,T))$ (resp.~$Z^1(k,W(G^*,T^*))$). Since $T$ and $T^*$ correspond as above,
the image of $f$ in $Z^1(k,W(G^*,T^*))$ under the isomorphism induced by $\delta_0$
is cohomologous to $f^*$. Moreover, by adjusting $g^*$ by an appropriate element of
$N_{G^*}(T_0^*)(k\sep)$, one can arrange for these cocycles to coincide.

To show that $T$ and $T^*$ are in $k$-duality,
define an isomorphism $\map\delta{\bX^*(T)}{\bX_*(T^*)}$ as follows.
An element of $\bX^*(T)$ can be written in the form
$\lsup{g}\chi$ for a unique $\chi\in\bX^*(T_0)$.
Let
\begin{equation}
\label{eq:delta}
\delta (\lsup{g}\chi) = \lsup{g^*}(\delta_0\chi).
\end{equation}
It is easily verified that $\delta$ satisfies (\ref{cond:roots-to-coroots}).

We need to show that $\delta$ is $\Gal(k)$-equivariant.
Suppose $\sigma\in \Gal(k)$.  Then 
$$
\delta(\sigma(\lsup g\chi))
= \delta\bigl(\lsup{\sigma (g)}(\sigma\chi)\bigr)
= \delta\bigl(\lsup{g g\inv \sigma (g)}(\sigma\chi)\bigr)
= \lsup {g^*}\bigl(\delta_0(\lsup{g\inv \sigma (g)}(\sigma\chi))\bigr)
= \lsup {g^*}\bigl(\delta_0(\lsup{f(\sigma)}(\sigma\chi))\bigr).
$$
On the other hand,
$$
\sigma(\delta(\lsup g\chi)) = \sigma(\lsup{g^*}(\delta_0\chi))
= \lsup{\sigma (g^*)}(\sigma(\delta_0\chi))
= \lsup{\sigma (g^*)}(\delta_0(\sigma \chi)).
$$
It follows that $\delta$ will be $\Gal(k)$-equivariant if and only if
$$
\delta_0(\lsup{f(\sigma)}(\sigma\chi)) = 
\lsup{{g^*}\inv\sigma (g^*)}(\delta_0(\sigma \chi))
= \lsup{f^*(\sigma)}(\delta_0(\sigma \chi))
$$
for all $\sigma\in\Gal(k)$ and $\chi\in\bX^*(T_0)$.
But this equality follows from the equivariance of $\delta_0$ with respect to the
action of $W(G,T_0) = W(G^*,T_0^*)$ (see Remark~\ref{rem:weyl-isomorphism})
and the fact that 
$f^*(\sigma) = \delta_0(f(\sigma))$.
Thus, $\delta$ is a duality map.
Finally, note that varying the choices of the above elements $g$ and $g^*$ has the effect
of altering $\delta$ by an element of $W(G,T)^{\Gal(k)}$.
\end{proof}

\begin{rem}
\label{rem:tori-correspondence}
Given a maximal $k$-torus $S\subseteq G$, let $S^*\subseteq G^*$
correspond to $S$ via Proposition \ref{prop:tori-duality}.
Repeating the proof with $S$ and $S^*$ in place of $T_0$ and $T_0^*$, we see
the following.
If the maximal $k$-torus $T\subseteq G$ corresponds to $T^*\subseteq G^*$,
and we write $T = \lsup g S$ with $g\in G(k\sep)$,
then there exists $g^* \in G^*(k\sep)$ such that
$T^* = \lsup{g^*}S^*$, and
the cocycle that sends $\sigma\in\Gal(k)$ to the image in $W(G,S)$
of $g\inv \sigma(g)$ corresponds to the analogous cocycle determined by $g^*$
under the identification between $W(G,S)$ and $W(G^*,S^*)$.
Thus, this latter identification leads to the same correspondence
$\TTst(G,k) \longleftrightarrow \TTst(G^*,k)$ of Proposition \ref{prop:tori-duality}.
Moreover, if $\delta_S$ and $\delta_T$ are duality maps for $S$ and $T$,
then there is some $w\in W(G,S)^{\Gal(k)}$ such that
for all $\chi \in \bX^*(S)$, we have
$\lsup{g^* w} \delta_S(\chi) = \delta_T(\lsup g \chi)$.
Since $\chi\mapsto \lsup w \delta_S(\chi)$ is another duality map for $S$,
we may replace $\delta_S$ by it
and then we can write
$\lsup{g^*} \delta_S(\chi) = \delta_T(\lsup g \chi)$
in analogy with (\ref{eq:delta}).
\end{rem}

\section{The conorm map}
\label{sec:conorm}
In this section, let $\tilde{G}$ and $G$ be quasisplit connected reductive $k$-groups.
Let
$\Gamma$ be a finite group,
and suppose that $(\phi,j_*)$ is a parascopic datum
for $(\tilde{G},\Gamma,G)$
relative to the maximal $k$-tori
$\tilde T\subseteq\tilde G$ and $T\subseteq G$.

\begin{rem}
\label{rem:nonquasisplit}
We assume that our groups are quasisplit only to assure that they have duals satisfying
Proposition \ref{prop:tori-duality}.
But such duals exist for some other groups (a matter that we will take up elsewhere),
and in such cases we only need to make the weaker assumption that $G$ is parascopic for $(\tilde{G},\Gamma)$
via the datum $(\phi,j_*)$.
\end{rem}

\begin{rem}
\label{rem:square}
Given a maximal $k$-torus $S\subseteq G$,
one has from
Proposition
\ref{prop:parascopy-equivalence}(\ref{item:all-tori-equally-good},\ref{item:Ttilde-stably-determined})
a maximal $k$-torus $\tilde S \subseteq \tilde{G}$,
uniquely determined up to stable conjugacy.
By Proposition \ref{prop:tori-duality},
$S$ and $\tilde S$ determine
maximal $k$-tori
$S^* \subseteq G^*$
and
$\tilde{S}^* \subseteq \tilde{G}^*$
(up to stable conjugacy),
and duality maps
$\map\delta{\bX^*(S)}{\bX_*(S^*)}$
and
$\map{\tilde\delta}{\bX^*(\tilde{S})}{\bX_*(\tilde{S}^*)}$
(up to conjugacy by
$W(G,S)^{\Gal(k)}$ and $W(\tilde{G},\tilde{S})^{\Gal(k)}$, respectively).
Similarly, given a maximal $k$-torus $S^*\subseteq G^*$, one obtains
$S\subseteq G$ and thus all of the other data above.
\end{rem}

Using $T$ and $\tilde{T}$,
choose
$T^*$,
$\tilde{T}^*$,
$\delta$,
and $\tilde\delta$
as in the preceding Remark.
From
\S\ref{sec:modules},
the parascopic datum $(\phi,j_*)$ determines a
$\Gal(k)$-equivariant
map $\map{\normchar = \normchar[T]}{\bX^*(T)}{\bX^*(\tilde T)}$.
Define
$$
\map{\dnormcochar[T^*,] 
:=\tilde\delta\circ\normchar\circ\delta\inv}{\bX_*(T^*)}{\bX_*(\tilde T^*)}.
$$
Then $\dnormcochar[T^*,]$ determines a \emph{conorm} homomorphism
$$
\map{\dnorm[T^*]}{T^*}{\tilde T^*}.
$$
Since both $\delta$ and $\tilde\delta$ are
$\Gal(k)$-equivariant, so is $\dnormcochar[T^*,] $.  Hence $\dnorm[T^*]$ is defined over $k$.
We also have a corresponding map $\map{\dnormchar[T^*]}{\bX^*(\tilde T^*)}{\bX^*(T^*)}$.
More explicitly,
$$
\dnormchar[T^*] = {\delta^*}\inv\circ\normcochar\circ\tilde\delta^*,
$$
where $\map{\normcochar}{\bX_*(\tilde T)}{\bX_*(T)}$ is the adjoint of $\normchar$.

From Remark \ref{rem:weyl-isomorphism},
the duality maps $\delta$ and $\tilde\delta$
determine
identifications $W(G,T) \longrightarrow W(G^*,T^*)$
and $W(\tilde G,\tilde T) \longrightarrow W(\tilde G^*,\tilde T^*)$.
Thus, the embedding $i$ of $W(G,T)$
in $W(\tilde G,\tilde T)$ (see \S\ref{sec:weyl-embedding}) determines an embedding
(which we will also denote by $i$)
of $W(G^*,T^*)$ in $W(\tilde G^*,\tilde T^*)$.

\begin{rem}
\label{rem:weyl-equivariant}
If we identify $W(G^*,T^*) = W(G,T)$ with its image in
$W(\tilde G^*,\tilde T^*) = W(\tilde G,\tilde T)$ under $i$,
it is clear that
$\normchar$ is $W(G,T)$-equivariant.
Since $\delta$ and $\tilde\delta$ are
$W(G,T)$-equivariant (given the identifications in the preceding paragraph),
it follows that $\dnorm[T^*]$ is as well.
\end{rem}

Let $s\in T^*(\bar k)$ and let $\tilde s = \dnorm[T^*](s)\in\tilde T^*(\bar k)$.  Let 
$H^* = C_{G^*}(s)$ and $\tilde H^* = C_{\tilde G^*}(\tilde s)$.

\begin{prop}
\label{prop:centralizer-weyl}
The embedding $i:W(G^*,T^*)\longrightarrow W(\tilde G^*,\tilde T^*)$
restricts to give embeddings of $W(H^*,T^*)$ in $W(\tilde H^*,\tilde T^*)$
and of $W(\Hsc,T^*)$ in $W(\Htsc,\tilde T^*)$.
\end{prop}
\begin{proof}
Suppose that $w\in W(H^*,T^*)$.  Then by Remark~\ref{rem:weyl-equivariant},
$$
(i(w))(\tilde s) = (i(w))(\dnorm[T^*](s)) = \dnorm[T^*](w(s)) = 
\dnorm[T^*](s) = \tilde s.
$$
Thus $i$ gives an embedding of $W(H^*,T^*)$ in $W(\tilde H^*,\tilde T^*)$.

Now suppose $w\in W(\Hsc,T^*)$.  According to~\cite{carter:finite}*{Theorem 3.5.3},
$w$ is a product of reflections through roots $\alpha^*\in\Phi(G^*,T^*)$
such that $\alpha^*(s) = 1$.  For such a root $\alpha^*$, let $\alpha$ be the corresponding
root in $\Phi(G,T)$:
$\alpha = \delta\inv ({\alpha^*}\spcheck)$.
Then 
$$
i(w_{\alpha^*}) = i(w_\alpha) = \prod_{\beta\in\Xi}w_\beta
$$
in the notation of Proposition~\ref{prop:weyl-embedding}.
For each $\beta\in\Xi$, 
let $\beta^*$ be the corresponding root in $\Phi(\tilde G^*,\tilde T^*)$:
$\beta^* = {\tilde\delta}^{*\,\inv} (\beta\spcheck)$.
Then $i(w_{\alpha^*}) = \prod_\Xi w_{\beta^*}$, and by \loccit, 
to show that $i(w)$ lies in $W(\Htsc,\tilde T^*)$, it suffices to show that 
$\beta^*(\tilde s) = 1$
for all $\beta\in\Xi$.
For such a root $\beta^*$, 
$$
\beta^*(\tilde s) = \beta^*(\dnorm[T^*](s)) = (\dnormchar[T^*]\beta^*)(s),
$$
so it is enough to show that $\dnormchar[T^*]\beta^*$ is an integer multiple of $\alpha^*$.

Recall the identifications of $V_*(T)$ with $V_*(\tilde{T})^\Gamma$ and
$V^*(T)$ with $V^*(\tilde{T})^\Gamma$ described at the beginning of \S\ref{sec:weyl-embedding}.
We have
\begin{align*}
\normcochar[T,]\beta^\vee &= \sum_{\gamma\in\Gamma}\gamma\cdot\beta^\vee
&& \text{(by (\ref{eq:cochar-norm}))}\\
& = |\stab_\Gamma\beta|\sum_{\beta'\in\Xi} (\beta')^\vee
&& \text{(by Remark~\ref{rem:psi})}\\
& = \frac{|\Xi||\stab_\Gamma\beta|}{|\Gamma\cdot \tilde\alpha|}\alpha^\vee
&& \text{(by Lemma~\ref{lem:multiple})}.
\end{align*}
But the constant $|\Xi||\stab_\Gamma\beta|/|\Gamma\cdot \tilde\alpha|$ is always
integral.  Indeed, (in the terminology of~\S\ref{sec:weyl-embedding}) in case (1), we have
$\Xi = \Gamma\cdot \tilde\alpha$, while in case (2), 
$|\Gamma\cdot \tilde\alpha| = 2|\Xi|$ and $|\stab_\Gamma\beta|$ is even.
Translating this to the dual setting, we have that $\dnormchar\beta^*$ is an integer multiple of
$\alpha^*$, and the proposition follows.
\end{proof}

\section{Lifting of semisimple geometric conjugacy classes}
\label{sec:geometric-general}
We now prove Statement \eqref{stmt:geom-conj-lift}
from the Introduction.
Let $\tilde{G}$, $G$, $\Gamma$, 
$\tilde{T}$, $T$,
and 
$(\phi,j_*)$
be as in \S\ref{sec:conorm}.

Given a maximal $k$-torus $S^*$ of $G^*$, choose corresponding maximal $k$-tori $\tilde S$,
$S^*$, and $\tilde S^*$, and and duality maps $\delta_S$ and $\delta_{\tilde S}$ as in Remark~\ref{rem:square}.
From Proposition \ref{prop:parascopy-equivalence}(\ref{item:all-tori-equally-good}),
there exists a parascopic datum
$(\phi',j_*')$
for $(\tilde{G},\Gamma,G)$ with respect to $\tilde{S}$ and $S$ that is equivalent to $(\phi,j_*)$.
Thus from~\S\ref{sec:conorm}, corresponding to these arbitrary, implicit choices, we have
a $W(G^*,S^*)$-equivariant
$k$-morphism $\map{\dnorm[S^*]}{S^*}{\tilde{S}^*}$.

\begin{prop}
\label{prop:lift-geometric-general}
There is a canonical
$k$-morphism $\dnorm$ from the $k$-variety of geometric semisimple conjugacy classes in $G^*$
to the analogous variety for $\tilde{G}^*$.  
Moreover, if $S^*$ is a maximal $k$-torus of $G^*$ and $s\in S^*(\bar k)$, then 
$$
\dnorm[S^*](s)\in\dnorm([s])(\bar k),
$$
where $[s]$ is the geometric conjugacy class of $s$ in $G^*$.
That is, $\dnorm$ is compatible with the conorms
$\dnorm[S^*]$ on all maximal $k$-tori in $G^*$.
\end{prop}

\begin{proof}
Let $S^*$ be a maximal $k$-torus in $G^*$.
As noted above, implicit in the construction of $\dnorm[S^*]$ are $k$-tori
$S$, $\tilde{S}$, and $\tilde{S}^*$, and duality maps $\delta_S$ and
$\delta_{\tilde S}$ (as in Remark \ref{rem:square}), as well as 
a parascopic datum $(\phi',j_*')$ for $(\tilde{G},\Gamma,G)$
with respect to $\tilde{S}$ and $S$ that is equivalent to $(\phi,j_*)$.
Since $\map{\dnorm[S^*]}{S^*}{\tilde{S}^*}$ is a $W(G^*,S^*)$-equivariant
$k$-morphism,
we obtain a $k$-morphism
$$
\map{\dnorm}{S^*/W(G^*,S^*)}{\tilde S^*/W(\tilde G^*,\tilde S^*)}.
$$
But these latter two varieties are $k$-isomorphic to the varieties
of geometric semisimple conjugacy classes in $G^*$ and $\tilde G^*$, respectively.

From the maximal $k$-tori $T\subseteq G$ and $\tilde{T}\subseteq \tilde{G}$,
Proposition \ref{prop:tori-duality}
gives us maximal $k$-tori
$T^*\subseteq G^*$,
and $\tilde{T}^*\subseteq \tilde{G}^*$, and duality maps $\delta_T$ and $\delta_{\tilde T}$.
We now show that $\dnorm$ is independent of the choice of the torus $S^*$
by showing that we would have obtained the same map had we chosen $S^* = T^*$
and $(\phi',j'_*) = (\phi,j_*)$.

Suppose $(\phi',j_*')$ is equivalent to $(\phi,j_*)$ via
$g\in G(k\sep)$ and $\tilde g\in\tilde G(k\sep)$. Then
$\lsup gT = S$, $\lsup {\tilde{g}} \tilde{T} = \tilde{S}$, 
$\phi'(\gamma) = \Int(\tilde{g})_* \circ \phi(\gamma) \circ \Int(\tilde{g})_*\inv$
for all $\gamma\in \Gamma$,
and
$j_*' = \Int(\tilde{g})_* \circ j_* \circ \Int(g)_* \inv$.
It follows from this and the
definition of $\normchar$
in \S\ref{sec:modules}
that
\begin{equation}
\label{eq:left-face}
\normchar[S] = \Int(\tilde g^{-1})^*\circ\normchar[T]\circ\Int (g)^*.
\end{equation}

Choose $g^*\in G^*(k\sep)$ and $\tilde g^*\in\tilde G^*(k\sep)$ such that $\lsup {g^*}T^* = S^*$
and $\lsup{\tilde g^*}\tilde T^* = \tilde S^*$ and such that
$g^*$ and $\tilde g^*$ are compatible with $g$ and $\tilde g$ (respectively)
as in Remark~\ref{rem:tori-correspondence}.  
Then there exist duality maps $\delta_T$ and $\delta_{\tilde T}$ such that
the top and bottom faces of the following diagram commute:
$$
\begin{xy}
\xymatrix{ 
& \bX^*(\tilde{S}) \ar[rr]^{\delta_{\tilde S}}\ar@{<-}'[d][dd]_(.3){\normchar[S]}
& & \bX_*(\tilde S^*) \ar@{<-}[dd]^{\dnormcochar[S^*,]}
\\ 
\bX^*(\tilde T) \ar[ur]^{\Int(\tilde g\inv)^*}\ar[rr]^(.65){\delta_{\tilde T}}\ar@{<-}[dd]_{\normchar[T]}
& & \bX_*(\tilde{T}^*)  \ar[ur]^{\Int(\tilde g^*)_*}\ar@{<-}[dd]^(.35){\dnormcochar[T^*,]}
\\ 
& \bX^*(S) \ar'[r]_(.6){\delta_S}[rr]
& & \bX_*(S^*)
\\ 
\bX^*(T)  \ar[rr]_{\delta_T}\ar[ur]^{\hspace{.7cm}\Int(g\inv)^*}
& & \bX_*(T^*)  \ar[ur]^{\Int(g^*)_*} 
} 
\end{xy}
$$
The front and back faces also commute by the definitions of 
$\dnormcochar[T^*,]$ and $\dnormcochar[S^*,]$.  
The left face commutes by (\ref{eq:left-face}).
Since all of the horizontal maps are isomorphisms, the right face must also commute.
That is, we have the equality
$$
\dnormcochar[S^*,] = \Int(\tilde g^*)_*\circ\dnormcochar[T^*,]\circ\Int(g^*)\inv_*
$$
of maps $\bX_*(S^*)\longrightarrow\bX_*(\tilde S^*)$.
Therefore, the corresponding homomorphisms $S^*\longrightarrow\tilde S^*$ must
be equal, i.e.,
\begin{equation}
\label{eq:torus-independent}
\dnorm[S^*] = \Int(\tilde g^*)\circ\dnorm[T^*]\circ\Int(g^*)\inv.
\end{equation}
It follows immediately that the definition of $\dnorm$ is
independent of the particular choice of torus $S^*$.
\end{proof}

\section{Main Theorem}
\label{sec:main}
Let $\tilde{G}$, $G$, $\Gamma$, 
$\tilde{T}$, $T$,
and 
$(\phi,j_*)$
be as in \S\ref{sec:conorm}
and \S\ref{sec:geometric-general}.

\begin{thm}
\label{thm:main}
Let $s_1,s_2\in G^*(k)$ be semisimple elements that are stably conjugate,
and let
$T_i^*$ be maximal $k$-tori in $G^*$ containing $s_i$.
Then the elements $\dnorm[T^*_i](s_i) \in \tilde{G}^*(k)$
are stably conjugate.
\end{thm}

\begin{rem}
\label{rem:main}
That is,
there is a well-defined map $\dnormst$ from
the set of semisimple
stable conjugacy classes in $G^*(k)$
to
the set of such classes in $\tilde{G}^*(k)$,
such that for every semisimple $s \in G^*(k)$ and every
maximal $k$-torus $T^*$ in $G$ containing $s$,
we have that
$\dnorm[T^*](s) \in \dnormst([s]_{\text{st}})$,
where $[s]_{\text{st}}$ is the stable conjugacy class of $s$.
\end{rem}

\begin{proof}[Proof of Theorem \ref{thm:main}]
From $T_i^*$,
choose $T_i$, $\delta_i$, $\tilde T_i$,
$\tilde{T}_i^*$,
and
$\tilde\delta_i$
as in Remark \ref{rem:square}, and a parascopic datum $(\phi_i,j_{i,*})$
for $(\tilde G,\Gamma,G)$ with respect to $\tilde T_i$ and $T_i$ that is equivalent
to $(\phi,j_*)$. (These choices are implicit in the definition of $\dnorm[T_i^*]$.)
Let $\tilde s_i = \dnorm[T_i^*](s_i)\in\tilde G^*(k)$.
According to 
Proposition \ref{prop:lift-geometric-general},
$\tilde s_1$ is geometrically conjugate to $\tilde s_2$
in $\tilde G^*$.
We want to show that the stable conjugacy classes of
$\tilde s_1$ and $\tilde s_2$ in $\tilde G^*(k)$ coincide.

Let $H^*_1= C_{G^*}(s_1)$ and $\tilde H^*_1= C_{\tilde G^*}(\tilde s_1)$.
We have that $\lsup{g^*}s_1 = s_2$ for some $g^*\in G^*(k\sep)$ such that
${g^*}\inv\sigma(g^*)\in {H^*_1}\conn(k\sep)$ for all $\sigma\in\Gal(k)$.  Moreover,
by replacing $g^*$ by $g^*h^*$ for an appropriate element $h^*\in {H^*_1}\conn(k\sep)$,
we may assume, in addition, that $\lsup{g^*}T^*_1 = T^*_2$.
It follows that ${g^*}\inv\sigma(g^*)\in N_{{H^*_1}\conn}(T^*_1)(k\sep)$
for all $\sigma\in\Gal(k)$.

Choose $g\in G(k\sep)$ compatible with $g^*$ as in Remark~\ref{rem:tori-correspondence}
such that $\lsup g T_1 = T_2$.
By Proposition~\ref{prop:parascopy-equivalence}(\ref{item:all-tori-equally-good}) and Lemma~\ref{lem:equivalence}, there
exists $\tilde g\in\tilde G(k\sep)$ such that $\lsup{\tilde g}\tilde T_1 = \tilde T_2$ and
$i\bigl(\overline {g\inv\sigma(g)}\bigr) = \overline{\tilde{g}\inv\sigma(\tilde{g})}$.
Choose $\tilde g^*\in G(k\sep)$ compatible with $\tilde g$ as in Remark~\ref{rem:tori-correspondence}
such that $\lsup{\tilde g^*} \tilde T_1^* = \tilde T_2^*$.

Since $(\phi_1,j_{1,*})$ and $(\phi_2,j_{2,*})$ are equivalent to $(\phi ,j_*)$, they are
equivalent to each other.
By Lemma~\ref{lem:equivalence}, $g$ and $\tilde g$ implement an equivalence of $(\phi_1,j_{1,*})$ with
a parascopic datum for $(\tilde G,\Gamma,G)$ with respect to $\tilde T_2$ and $T_2$. The
latter datum must therefore be equivalent to $(\phi_2,j_{2,*})$, and hence
is related to $(\phi_2,j_{2,*})$ as in
Proposition~\ref{prop:parascopy-equivalence}(\ref{item:same-tori-equivalence}).
It follows that the conorms $\abmap{T_2^*}{\tilde T_2^*}$ corresponding to these data
differ by the action of an element of
$W(\tilde G^*,\tilde T_2^*)^{\Gal(k)}$. Thus the stable conjugacy classes of the images of $s_2$
under these conorms coincide. Hence we may assume that $(\phi_1,j_{1,*})$ and $(\phi_2,j_{2,*})$
are equivalent via the particular elements $g$ and $\tilde g$.

As in \eqref{eq:torus-independent},
we have
$\dnorm[T_2^*] = \Int(\tilde g^*)\circ\dnorm[T_1^*]\circ\Int(g^*)\inv$.
Thus
\begin{equation}
\label{eq:conorm}
\tilde s_2 = \dnorm[T_2^*](s_2) = \lsup{\tilde g^*}\bigl(\dnorm[T_1^*](\lsup {{g^*}\inv}s_2)\bigr)
 = \lsup{\tilde g^*}\bigl(\dnorm[T_1^*](s_1)\bigr) = \lsup{\tilde g^*}\tilde s_1.
 \end{equation}
It remains to show that $\tilde g^*{}\inv\sigma(\tilde g^*)\in \tilde H^*_1{}\conn(k\sep)$
for all $\sigma\in\Gal(k)$.
But $\tilde g^*$ was chosen so that the image of $\tilde g^*{}\inv\sigma(\tilde g^*)$ in $W(G^*,T^*_1)$ is
$i\bigl(\overline{g^*{}\inv\sigma(g^*)}\bigr)$.
Since $\overline{{g^*}\inv\sigma(g^*)}\in W({H^*_1}\conn,T^*_1)$, it follows that
$i\bigl(\overline{{g^*}\inv\sigma(g^*)}\bigr)$ lies in $W(\tilde H^*_1{}\conn,\tilde T_1^*)$
by Proposition~\ref{prop:centralizer-weyl}.
Therefore, $\tilde g^*{}\inv\sigma(\tilde g^*)\in N_{\tilde H^*_1{}\conn}(\tilde T^*_1)(k\sep)
\subseteq \tilde H^*_1{}\conn(k\sep)$, so $\tilde s_1$
and $\tilde s_2$ are stably conjugate.
\end{proof}

\begin{cor}
\label{cor:k-finite}
In the situation of Theorem \ref{thm:main} and Remark \ref{rem:main},
suppose that $k$ is perfect and has cohomologial dimension $\leq 1$.
Then $\dnorm$ refines to a map $\dnormst$ from the set of semisimple, $G^*(k)$-conjugacy classes in $G^*(k)$
to the set of semisimple, $\tilde{G}^*(k)$-conjugacy classes in $\tilde{G}^*(k)$.
\end{cor}


\begin{proof}
From the Lang-Steinberg Theorem
(see \cite{serre:galois}*{\S III.2.3}),
$H^1(k,M)$ is trivial for every connected reductive $k$-group $M$.
Thus, from the first paragraph of \cite{kottwitz:rational-conj}*{\S3},
we see that semisimple stable conjugacy classes and semisimple rational conjugacy classes
coincide in the group of $k$-points of a connected reductive $k$-group,
so our result now follows from Theorem \ref{thm:main}.
\end{proof}

\appendix

\section{When must a quasi-semisimple automorphism preserve a Borel-torus pair defined over $k\sep$?}
\label{sec:separable}

We now offer a proof promised in Remark \ref{rem:automatically-separable}.

\begin{lem}
\label{lem:rational-torus}
Suppose $\tilde{G}$ is a connected reductive $k$-group,
and $\gamma$ is a quasi-semisimple $k$-automorphism of $\tilde G$.
Suppose that at least one of the following holds:
\begin{enumerate}[(i)]
\item
\label{item:char-not-two}
the characteristic of $k$ is not two;
\item
$k$ is perfect;
\item
\label{item:not-bad-case}
no power of $\gamma$ acts via a non-inner automorphism on any factor of $\tilde{G}$ of type $A_{2n}$.
\end{enumerate}
Then $\gamma$ preserves a Borel-torus pair defined over $k\sep$.
\end{lem}
\begin{proof}
Note that this statement is obvious when $k$ is perfect.

We may assume that $k=k\sep$ is separably closed.  It is enough to show that $\tilde G$
has a $\gamma$-stable maximal $k$-torus $\tilde T_\bullet$ contained in
a $\gamma$-stable Borel subgroup $\tilde B_\bullet$.  For then
$\tilde T_\bullet$ must be split over $k$, so the root groups corresponding to the
roots in $\Phi(\tilde G,\tilde T_\bullet)$ must
be defined over $k$, and hence so must $\tilde B_\bullet$.

Fix a Borel $\bar k$-subgroup $\tilde B_*$ of $\tilde G$ and
a maximal $\bar k$-torus $\tilde T_*\subseteq \tilde B_*$.  The homogeneous space
$\tilde G/\tilde T_*$ can be viewed, via the map
$g\tilde T_*\mapsto (\lsup g \tilde B_*,\lsup g \tilde T_*)$,
as the $\bar k$-variety $X$ of all pairs $(\tilde B,\tilde T)$,
where $\tilde B$ is a Borel $\bar k$-subgroup of $\tilde G$, and $\tilde T$ is a maximal
$\bar k$-torus of $\tilde B$.  Taking $\tilde T_*$ to be defined over $k$ shows that
$X$ can be given the structure of a $k$-variety; this structure is easily seen to be independent of
the particular choice of $(\tilde B_*,\tilde T_*)$.

Assume for now that $\tilde T_*$ is defined over $k$.
There is an obvious action of $\gamma$ on $X$.
Let $X^\gamma$ be the (nonempty) $\bar k$-variety  of
$\gamma$-fixed points in $X$.  A point $x\in X^\gamma(\bar k)$ corresponds to a
pair $(\tilde B,\tilde T)$ as above such that $\tilde B$ and $\tilde T$ are $\gamma$-stable.
Moreover, if $x\in X(k)\cap X^\gamma(\bar k)$ is represented by $g\in\tilde G(\bar k)$, then
$\Int(g): \tilde T_*\longrightarrow \tilde T := \lsup{g}\tilde T_*$ is a $k$-isomorphism.  Thus $\tilde T$ (and
hence $\tilde B$) is defined over $k$~\cite{borel:linear}*{Cor.~14.5}.
It follows that the desired Borel-torus pair exists provided that 
$X(k)\cap X^\gamma(\bar k)\neq\emptyset$.
To prove this nonemptiness, it suffices to show that $X^\gamma$ is defined
over $k$, for then $X^\gamma(k) = X(k)\cap X^\gamma(\bar k)$ must be dense in 
$X^\gamma$~\cite{borel:linear}*{Cor.~13.3}.
By~\cite{springer:lag}*{Thm.~11.2.13},
this will follow in turn if we can prove the equality of tangent spaces
\begin{equation}
\label{eq:tangent-space}
\Tan_x(X^\gamma) = \Tan_x(X)^\gamma\quad\mbox{for all $x\in X^\gamma(\bar k)$.}
\end{equation}

Let $\tilde T_\bullet$ and $\tilde B_\bullet$ be as in Definition~\ref{defn:qss}.
The Bruhat decomposition can be used to express $\tilde G$ as a disjoint union
\begin{equation}
\label{eq:bruhat}
\coprod_{w\in W(\tilde G,\tilde T_\bullet)}\tilde U\tilde U_w n_w\tilde T_\bullet  \,,
\end{equation}
where $n_w$ is a representative of $w$ in $N_{\tilde G}(\tilde T_\bullet)(\bar k)$, $\tilde U$ is the
unipotent radical of $\tilde B_\bullet$, and $\tilde U_w$ is the group generated by the root groups
corresponding to
$\alpha\in\Phi^-(\tilde G,\tilde T_\bullet)\cap w\Phi^+(\tilde G,\tilde T_\bullet)$, where
positivity is with respect to $\tilde B_\bullet$.

Let $z\in X(\bar k) = (\tilde G/\tilde T_\bullet)(\bar k)$ be the point corresponding to the
coset $\tilde T_\bullet$.
Since the obvious map
$\tilde U\times\tilde U_w\times\tilde T_\bullet\longrightarrow\tilde U\tilde U_w n_w\tilde T_\bullet$
is an isomorphism of varieties for each $w\in W(\tilde G,\tilde T_\bullet)$, 
it follows from (\ref{eq:bruhat}) that
\begin{equation}
\label{eq:fixed-coset}
(\tilde U\tilde U_w n_w\tilde T_\bullet\cdot z)^\gamma = 
\left\{
\begin{array}{ll}
\tilde U^\gamma\tilde U_w^\gamma n_w\cdot z & \mbox{if $w\in W(\tilde G,\tilde T_\bullet)^\gamma$,}\\
\emptyset & \mbox{otherwise.}
\end{array}
\right.
\end{equation}

From Lemma \ref{lem:steinberg},
$G = (\tilde G^\gamma)\conn$ is a reductive group,
$T_\bullet:= (\tilde T_\bullet^\gamma)\conn = G\cap\tilde T_\bullet$
is a maximal torus in $G$,
and $B_\bullet := (\tilde B_\bullet^\gamma)\conn$ is a Borel subgroup of $G$ containing $T_\bullet$.

Observe that $G$ acts on $X^\gamma$ via left translation.  It follows from (\ref{eq:fixed-coset})
that $X^\gamma$ is the union of a finite number of $G$-orbits represented by some
subset of $\{ n_w\}$.  (In fact, we show in 
Proposition~\ref{prop:weyl-embedding} that 
$W(G,T_\bullet)$ embeds in $W(\tilde G,\tilde T_\bullet)^\gamma$, implying that 
one can take as representatives for $G\backslash X^\gamma$ any 
subset of $\{ n_w\}$
corresponding to representatives of $W(G,T_\bullet)\backslash W(\tilde G,\tilde T_\bullet)^\gamma$ in
$W(\tilde G,\tilde T_\bullet)^\gamma$.)  Moreover, the stabilizer in $G$ of
$n_w\cdot z$ is precisely $T_\bullet$.  In particular, each $G$-orbit has dimension
$\dim G - \dim T_\bullet$.
Since these orbits are irreducible and of the same dimension as
$X^\gamma$, it is straightforward to show that they are precisely the irreducible components
of $X^\gamma$.  Since they are disjoint and finite in number,
they must also be the connected components of $X^\gamma$.
It therefore suffices to verify (\ref{eq:tangent-space}) for $x = n_w\cdot z$.

Since the dimension of any $G$-orbit in $X^\gamma$ is $\dim G - \dim T_\bullet$, we have
$\dim \Tan_x(X^\gamma) = \dim \Tan_x(G\cdot x)\geq \dim G - \dim T_\bullet$.
On the other hand, there is a natural identification
$\Tan_x(X) = L(\tilde G)/L(\tilde T_\bullet)$,
where $L$ denotes the Lie algebra functor,
so we have
$$
\Tan_x(X^\gamma)\subseteq \Tan_x(X)^\gamma = \left(L(\tilde G)/L(\tilde T_\bullet)\right)^\gamma
 = \Biggl(\bigoplus_{\alpha\in\Phi(\tilde G,\tilde T_\bullet)}L(\tilde U_\alpha)\Biggr)^\gamma.
$$
It follows from~\cite{steinberg:endomorphisms}*{\S8.2($2''''$)} that if
either assumption
(\ref{item:char-not-two})
or
(\ref{item:not-bad-case})
holds,
the last space is equal to
$$
\bigoplus_{\beta\in\Phi(G,T_\bullet)}L(U_\beta),
$$
which has dimension $|\Phi(G,T_\bullet)| = \dim G - \dim T_\bullet$.  Thus
$\Tan_x(X^\gamma) = \Tan_x(X)^\gamma$, as desired.
\end{proof}

\end{document}